\documentclass[11pt]{article}

\newcommand{\insieme}[1]{\left\{ #1 \right\}}
\usepackage{fullpage}
\usepackage{subcaption}
\usepackage{mathrsfs}  
\usepackage{bm}

\usepackage{tikz}
\usetikzlibrary{patterns}
\pgfdeclarepatternformonly{my crosshatch dots}{\pgfqpoint{-1pt}{-1pt}}{\pgfqpoint{6pt}{6pt}}{\pgfqpoint{8pt}{8pt}}%
{
    \pgfpathcircle{\pgfqpoint{0pt}{0pt}}{.5pt}
    \pgfpathcircle{\pgfqpoint{4pt}{4pt}}{.5pt}
    \pgfusepath{fill}
}

\pgfdeclarepatternformonly{nel}{\pgfqpoint{-1pt}{-1pt}}{\pgfqpoint{4pt}{4pt}}{\pgfqpoint{4pt}{4pt}}%
{
  \pgfsetlinewidth{0.3pt}
  \pgfpathmoveto{\pgfqpoint{0pt}{0pt}}
  \pgfpathlineto{\pgfqpoint{4pt}{4pt}}
  \pgfusepath{stroke}
}

\usepackage{enumerate}
\usepackage{amsmath,amsfonts,amssymb,amsthm, bbold}
\usepackage{svg}
\usepackage{graphics,esint}
\usepackage{epsfig}
\usepackage{xcolor}
\usepackage{xspace}
\definecolor{pingreen}{rgb}{0,39,14}

\usepackage{xfrac}
\usepackage{comment}
\usepackage{hyperref}
\usepackage{cleveref}

\crefname{section}{§}{§§}
\Crefname{section}{§}{§§}

\def\Acal{\mathcal{A}}

\def\rhs{r.h.s.\xspace}

\def\st{s.t.\xspace}

\makeatletter
\newtheorem*{rep@theorem}{\rep@title}
\newcommand{\newreptheorem}[2]{%
\newenvironment{rep#1}[1]{%
 \def\rep@title{#2~\ref{##1}}%
 \begin{rep@theorem}}%
 {\end{rep@theorem}}}
\makeatother

\newtheorem{theorem}{Theorem}[section]

\newtheorem{lemma}[theorem]{Lemma}
\newtheorem{definition}[theorem]{Definition}

\newtheorem{remark}[theorem]{Remark}

\newtheorem{proposition}[theorem]{Proposition}

\newtheorem{corollary}[theorem]{Corollary}

\newreptheorem{theorem}{Theorem}

\newcommand{\R}{\mathbb{R}}

\newcommand{\Fcal}{\mathcal{F}}

\newcommand{\FcalJ}{\tilde{\mathcal{F}}_{J,L,\eps}}
\newcommand{\Fcalt}{\mathcal{F}_{\tau,L,\eps}}

\def\T{\mathbb{T}}

\def\N{\mathbb N}
\def\eps{\varepsilon}

\def\per{\mathrm{Per}}

\def\eps{\varepsilon}

\def\d {\,\mathrm {d}}
\def\dx{\,\mathrm {d}x}
\def\dz{\,\mathrm {d}z}
\def\ds{\,\mathrm {d}s}

\def\dt{\,\mathrm {d}t}
\def\dy{\,\mathrm {d}y}
\def\da{\,\mathrm {d}a}
\def\db{\,\mathrm {d}b}

\def\loc{\mathrm{loc}}
\def\at{\alpha_{\eps,\tau}}
\def\Mi#1#2#3{\overline{\mathcal M}^{#3}_{\at}(u,x_{#3}^\perp,[#1,#2))}

\def\Gcal#1{\overline{\mathcal{G}}^{#1}_{\at,\tau}(u,x_{#1}^\perp,[0,L))}

\numberwithin{equation}{section}

\usepackage{authblk}

\author[1]{Sara Daneri\thanks{sara.daneri@gssi.it}}
\author[2]{Eris Runa\thanks{eris.runa@gmail.com}}
\affil[1]{Gran Sasso Science Institute, L'Aquila, Italy}
\affil[2]{Deutsche Bank, London, UK}

\title{One-dimensionality of the minimizers in the large volume  limit for a diffuse interface attractive/repulsive model in general dimension}

\date{}

\begin{document}
\maketitle

\begin{abstract}
   In this paper we consider the diffuse interface generalized antiferromagnetic model with local/nonlocal attractive/repulsive terms in competition studied in \cite{DKR}. The parameters of the model are denoted by $\tau$ and $\eps$: the parameter $\tau$ represents the relative strength of the local term with respect to the nonlocal one, while the parameter $\varepsilon$ describes the transition scale in the Modica-Mortola type term. 
   Restricting to a periodic box of size $L$, with $L$ multiple of the period of the minimal one-dimensional minimizers, in \cite{DKR} the authors prove that in any dimension $d\geq1$ and for small but positive $\tau$ and  $\varepsilon$ (eventually depending on $L$), the minimizers are non-constant one-dimensional periodic functions. 
   In this paper we prove that periodicity and one-dimensionality of minimizers occurs also in the zero temperature analogue of the  thermodynamic limit, namely as $L\to+\infty$. 
\end{abstract}

\section{Introduction}

In this paper we consider the following mean field free energy functional.
For $L,J,\eps>0$, $d\geq1$, $p\geq{d+2}$, $u\in W^{1,2}_{\loc}(\R^d;[0,1])$ and $[0,L)^d$-periodic, define

\begin{equation}\label{E:F}
\FcalJ(u):=\frac{J}{L^d}\Bigl[3\eps\int_{[0,L)^d}\|\nabla u(x)\|_1^2\dx+\frac{3}{\eps}\int_{[0,L)^d}W(u(x))\dx\Bigr]-\frac{1}{L^d}\int_{\R^d}\int_{[0,L)^d}|u(x+\zeta)-u(x)|^2K(\zeta)\dx\d\zeta,
\end{equation}
where, for $y=(y_1,\dots,y_d)\in\R^d$, $\|y\|_1=\sum_{i=1}^d|y_i|$, $W(t)=t^2(1-t)^2$ and $K(\zeta)=\frac{1}{(\|\zeta\|_1+1)^p}$.

In order to state our results properly, it is convenient to rescale the functional in order to have that the width of the optimal period for one-dimensional functions and their energy are of order $O(1)$.

For $\beta=p-d-1$ and $\tau>0$, setting 
\begin{align*}
&J_c=\int |\zeta_1|K(\zeta)\d\zeta,\quad J=J_c-\tau, \quad  x=\tau^{-1/\beta}\tilde x,\quad   \zeta=\tau^{-1/\beta}\tilde \zeta,\quad L=\tau^{-1/\beta}\tilde L,\\
& \tilde u(\tilde x)=u(x),\quad\FcalJ( u)=\tau^{1+1/\beta}\mathcal{F}_{\tau,\tilde L,\eps}(\tilde u)
\end{align*}
and finally dropping the tildas, one has that the rescaled functional has the form
\begin{equation}
\label{def:fcalt}
\Fcalt(u)=\frac{1}{L^d}\Bigl[\mathcal M_{\alpha_{\eps,\tau}}(u,[0,L)^d)\Bigl(\int_{\R^d}K_\tau(\zeta)|\zeta_1|\d\zeta-1\Bigr)-\int_{\R^d}\int_{[0,L)^d}|u(x)-u(x+\zeta)|^2K_\tau(\zeta)\dx\d\zeta\Bigr],
\end{equation}
where for $\alpha>0$
\begin{equation}
\label{eq:malpha}
\mathcal M_{\alpha}(u,[0,L)^d)=3\alpha\int_{[0,L)^d}\|\nabla u(x)\|_1^2\dx+\frac{3}{\alpha}\int_{[0,L)^d}W(u(x))\dx,
\end{equation}
$\alpha_{\eps,\tau}=\eps\tau^{1/\beta}$ and 
\begin{equation}
\label{def:Ktau}
K_\tau(\zeta)=\frac{1}{(\|\zeta\|_1+\tau^{1/\beta})^p}.
\end{equation}

 For fixed $\tau > 0$ and $\eps>0$, consider first for all $L > 0$ the minimal value obtained by $\Fcal_{\tau,L,\eps}$ on $[0,L)^d$-periodic one-dimensional functions (denoted by $\mathcal U^{per}_L$)
and then the minimal among these values as $L$ varies in $(0,+\infty)$. We will denote this value by $C^*_{\tau,\eps}$, namely

\begin{equation}\label{eq:cstartau}
	\begin{split}
		C^*_{\tau,\eps} := \inf_{L>0} \  \inf_{u\in \mathcal U^{per}_L} \Fcal_{\tau,L,\eps}(u). 
	\end{split}
\end{equation}

By the reflection positivity technique, in \cite{GLL1D} it is shown that such value is attained  by periodic one-dimensional functions with possibly infinite and not unique periods.

In \cite{DKR} we prove that,  for $\tau$ and $\eps$ sufficiently small, there exist periodic functions of finite period $2h$ for which the energy value $C^*_{\tau,\eps}$ is attained and the following property holds 

\begin{equation}\label{eq:refl}
	g(\nu+(2k+1)h+t)=1-g(\nu+(2k+1)h-t)\quad\text{ for all }k\in\N\cup\{0\},\,\, t\in[0,h].
\end{equation}

   We denote any of such finite optimal periods $2h$ (which may not be unique) as $2h^*_{\tau,\eps}$.

The main result obtained in \cite{DKR} is the following

\begin{theorem}[\cite{DKR}, Theorem 1.1]
	\label{Thm:DKR}
	Let $L=2kh^*_{\tau,\eps}$, $k\in\N$. Then there exist ${\tau}_L>0$, $\eps_L>0$ such that, for any $0<\tau\leq{\tau}_L$ and $0<\eps\leq \eps_L$ the minimizers of~\eqref{def:fcalt} are one-dimensional periodic functions of period $2h^*_{\tau,\eps}$.
\end{theorem}

In this paper, we prove that one-dimensionality and periodicity of minimizers of \eqref{def:fcalt} holds also in the zero temperature analogue of the thermodynamic limit, namely that the range of parameters  $\tau$, $\eps$ in which one-dimensionality and periodicity of minimizers is observed can be fixed independently on $L$. More precisely, the following holds.

\begin{theorem}\label{thm:main}
	Let $d\geq1$, $p\geq d+2$ and $h^{*}_{\tau,\eps}$ be the optimal period for fixed $\tau$. 
	Then there exists $\bar \tau$ and $\bar\eps$, such that for every $\tau< \bar \tau$ and $\eps<\bar \eps$, one has that for every $k\in \N$ and  $L = 2k h_{\tau,\eps}^{*}$,   the minimizers of $\Fcalt$ are optimal one-dimensional functions  of width $h_{\tau,\eps}^{*}$. 
\end{theorem}

\subsection{Scientific context}

 For the sharp interface limit of $\mathcal F_{\tau,L,\eps}$ as $\eps\to0$, namely

\begin{equation}\label{E:S}
	\mathcal{F}_{\tau,L}(E):=\frac{1}{L^d}\Bigl[\per_1(E;[0,L)^d)\Bigl(\int_{\R^d}K_\tau(\zeta)|\zeta_1|\d\zeta-1\Bigr)-\int_{\R^d}\int_{[0,L)^d}|\chi_E(x)-\chi_E(x+\zeta)|K_\tau(\zeta)\dx\d\zeta\Bigr],
\end{equation}
where $E\subset\R^d$, $d\geq2$,  one-dimensionality and periodicity of minimizers in the thermodynamic limit has been proved in \cite{GS} in the discrete setting (for exponents $p>2d$) and in \cite{DR} in the continuous one (for exponents $p\geq d+2$). In \cite{Ker} the results of \cite{DR} have been recently extended to a small range of exponents below $p=d+2$. 

For the most physically relevant exponents such as $p=d+1$ (thin magnetic films), $p=d$ (3D micromagnetics) and $p=d-2$ (diblock copolymers), a rigorous proof of pattern formation in dimension $d\geq2$ is still a challenging open problem. The main difficulty, which in \cite{GS},\cite{GR},\cite{DR} and \cite{Ker} is resolved for higher exponents $p$, is to prove that symmetry breaking occurs, namely that minimizers of \eqref{E:S} have less symmetries than the functional itself. In this case symmetry under coordinate permutations is lost, but if one considers the functional analogous to \eqref{E:S} where the $1$-norm is substituted by the Euclidean norm, full rotational symmetry loss is expected to occur as well. Another family of kernels which is physically relevant and widely used in the literature is the Yukawa or screened Coulomb kernel (see e.g. {\cite{BBCH, Bores, CCA, GCLW, IR}} ). For this type of kernels one-dimensionality and periodicity of minimizers in the thermodynamic limit has been proved in \cite{DR2}.  

The diffuse interface counterpart of \eqref{E:S}, namely \eqref{def:fcalt}, is expected to be the most physical one due to the presence of continuous phase transitions (see e.g. the famous model for block copolymers introduced by Ohta and Kawasaki in \cite{OK}). 
For the diffuse interface problem, even less results are available in the literature on the structure of minimizers. 
This is due to the fact that in this setting the geometry of possible  phase transitions is much richer. In particular,  other phenomena such as small amplitude oscillations and slow transitions may occur, adding mathematical difficulty to the problem and requiring  new estimates. In dimension $d=1$, in \cite{RenWei} the authors show that close to the local minima of the corresponding sharp interface problem there are local minima of the approximating diffuse one. In \cite{GLL1D}, the authors show that the constant $C^*_{\tau,\eps}$ in the one-dimensional problem is attained on periodic functions of possibly infinite period.  
The only result in dimension $d\geq2$ concerning one-dimensionality and periodicity of minimizers is the one given in \cite{DKR}. In \cite{DKR} one-dimensionality and periodicity of minimizers for $d\geq2$ was proved (see Theorem \ref{Thm:DKR}), for $L$ multiple of an optimal admissible  period $h^*_{\tau,\eps}$ and a range of positive $\tau,\eps$  depending on $L$. 

The aim of this paper is to show that a range of parameters for which pattern formation is observed can be chosen independently of $L$, no matter how large $L$ is (see Theorem \ref{thm:main}). 

The general strategy in order to choose an $L$-independent range of parameters is analogous to the one introduced in \cite{DR}[Section 7], involving:
 \begin{itemize}
 	\item a localization of the functional \eqref{def:fcalt} on small cubes of size $l<L$;
 	\item a decomposition of the functional into localized terms penalizing deviations from being one-dimensional in different ways;
 	\item a  partition of $[0,L]^d$ into sets $A\cup A_{1}\cup\dots\cup A_d$ where $z\in A_i$ if on the cube $Q_l(z)$ the function $u$ is $L^1$-close to stripes with boundaries orthogonal to $e_i$;
 	\item rigidity, stability and one-dimensional optimization arguments on slices in order to show that whenever $\tau=\tau(l)$ and $\eps=\eps(l)$ are sufficiently small,  $[0,L)^d=A_i$ for some $i\in\{1,\dots,d\}$.
 	\item when $[0,L)^d=A_i$, hence the symmetry is broken, conclude with a stability argument that $u=u(x_i)$. 
 \end{itemize}
However, not only the above-mentioned  rigidity, stability and one-dimensional optimization arguments are necessarily different from the ones of \cite{DR} due to the fact that the space of competitors is now given by functions instead of sets, but also the decomposition of the functional into localized terms is now different. In particular, a clever reformulation of the functional \eqref{def:fcalt} together with a disintegration argument are introduced in order to identify the contributions of $u$ on arbitrary intervals to the energy of the  one-dimensional slices. As a consequence, also the localization of the estimates introduced  in \cite{DKR} is much more delicate than the localization of the corresponding estimates for the sharp interface problem in \cite{DR}. As a counterpart, we believe that in their localized version given in this paper the estimates of \cite{DKR} are able to grasp in detail the mechanisms leading to the increase of the energy under the various possible deviations from one-dimensional profiles.

Regarding further fields of interest for pattern formation under attractive/repulsive forces in competition, we mention the following. Evolution problems of gradient flow type related to functionals with attractive/repulsive nonlocal terms in competition, both in presence and in absence of diffusion, are also well studied (see e.g. \cite{CCH,CCP, CDFS,DRR,craig, CT}). In particular, one would like to show convergence of the gradient flows or of their deterministic particle  approximations to configurations which are periodic or close to periodic states.
Another interesting direction would be to extend our rigidity results to non-flat surfaces without interpenetration of matter as investigated for rod and plate theories  in \cite{KS, LMP, OR}.

\section{Notation and preliminaries}

In the following, let $\N=\{1,2,\dots\}$, $d\geq 1$. 
Let $(e_1,\dots,e_d)$ be the canonical basis in $\R^d$ and for $y\in\R^d$ let $y_i=\langle y,e_i\rangle$ and $y_i^\perp:=y-y_ie_i$, where $\langle\cdot,\cdot\rangle$ is the Euclidean scalar product. 
For $y\in\R^d$, we denote by $\|y\|_1=\sum_{i=1}^d|y_i|$ its $1$-norm and we define $\|y\|_\infty=\max_i|y_i|$.
With a slight abuse of notation, we will sometimes identify $y^\perp_i\in[0,L)^d$ with its projection on the subspace orthogonal to $e_i$ or as an element of $\R^{d-1}$. 


For $z\in[0,L)^d$ and $r>0$, we also define 
\begin{equation*}
	\begin{split}
		Q_r(z)=\{x\in\R^d:\,\|x-z\|_\infty\leq r\} \qquad\text{and}\qquad Q_{r}^{\perp}(x^\perp_{i}) = \{z^\perp_{i}:\, \|x^{\perp}_{i} - z^{\perp}_{i} \|_\infty \leq r  \}. 
	\end{split}
\end{equation*}


For every $i\in\{1,\dots,d\}$ and for all $x_i^\perp\in[0,L)^{d-1}$, we define the slices of $u$ in direction $e_i$ as
\[
u_{x_i^\perp}:\R\to[0,1],\quad u_{x_i^\perp}(s):=u(s e_i+x_i^\perp).
\]

Notice that whenever $u\in W^{1,2}_{\loc}(\R^d;\R)$ then $u_{x_i^\perp}\in  W^{1,2}_{\loc}(\R; \R)$ for almost every $x^\perp_i$. 
We denote by $\partial_i$ the partial derivatives of a function with respect to $e_i$, $i\in\{1,\dots,d\}$.

Given a measurable set $A\subset\R^k$ with $k\in\{1,\dots,d\}$, we denote by $|A|$ its $k$-dimensional Lebesgue measure (or if A is contained in some $k$-dimensional plane of $\R^d$, its Hausdorff $k$-dimensional measure), being always clear from the context which will be the dimension $k$.

Moreover, let $\chi_A:\R^d\to\R$ be the function defined by
\begin{equation*}
	\chi_A(x)=\left\{\begin{aligned}
		&1 && &\text{if $x\in A$}\\
		&0 && &\text{if $x\in\R^d\setminus A$.}
	\end{aligned}\right.
\end{equation*}

A set $E\subset\R^d$ is of (locally) finite perimeter if the distributional derivative of $\chi_E$ is a (locally) finite measure.
We denote by $\partial E$ be the reduced boundary of $E$ and by $\nu^E$ the exterior normal to $E$.

Then one can define the $1$-perimeter of a set relative to $[0,L)^d$ as

\[
\per_1(E,[0,L)^d):=\int_{\partial E\cap [0,L)^d}\|\nu^E(x)\|_1\d\mathcal H^{d-1}(x)
\]
where $\mathcal H^{d-1}$ is the $(d-1)$-dimensional Hausdorff measure.

By extending the classical Modica-Mortola result~\cite{MM} to the anisotropic norm $\|\cdot\|_1$ and by continuity of the nonlocal term in~\eqref{E:F} with respect to $L^1$ convergence of functions valued in $[0,1]$, one has the following

\begin{theorem}
	\label{thm:gammaconv}
	As $\eps\to0$, the functionals $\Fcalt$ $\Gamma$-converge in $BV_{\loc}(\R^d;[0,1])$ to the functional
	\begin{equation}\label{eq:flim}
		\mathcal{F}_{\tau,L}(u):=\left\{
		\begin{aligned}
			&\frac{1}{L^d}\Bigl[\per_1(E;[0,L)^d)\Bigl(\int_{\R^d}K_\tau(\zeta)|\zeta_1|\d\zeta-1\Bigr)\\
			&\quad-\int_{\R^d}\int_{[0,L)^d}|\chi_E(x)-\chi_E(x+\zeta)|K_\tau(\zeta)\dx\d\zeta\Bigr] && &\text{if $u=\chi_E$}\\
			&+\infty && &\text{otherwise.}
		\end{aligned}
		\right.
	\end{equation}
\end{theorem}

Notice that the constant $3$ in~\eqref{E:F} is chosen in such a way that
\[
6\int_0^1t(1-t)\dt=1,
\]
so that the constant in front of the $1$-perimeter in~\eqref{eq:flim} is equal to $1$.

The kernel $K_\tau$ is, as shown in~\cite{DR}, reflection positive, namely it satisfies the following property: the function
\begin{equation*}
	\begin{split}
		\widehat K_\tau(t) := \int_{\R^{d-1}} K_\tau(t, \zeta_2,\ldots,\zeta_d)  \d\zeta_2\cdots\d\zeta_d. 
	\end{split}
\end{equation*}

is the Laplace transform of a nonnegative function.

Regarding the limit functional~\eqref{eq:flim}, we recall the following result, obtained in~\cite{DR} for $p\geq d+2$ and extended to a range of exponent below $d+2$ in \cite{Ker}.

\begin{theorem}
	\label{Thm:DR}
	Let $d\geq1$, $L>0$. Then, there exists $0<\bar{\alpha}<1$ and  $\tilde{\tau}_L>0$ such that, for all $p=d+2-\alpha$ with $\alpha\leq\bar \alpha$ and for all $0<\tau\leq\tilde{\tau}_L$, the minimizers of the functional $\mathcal{F}_{\tau,L}$ in~\eqref{eq:flim} are periodic unions of stripes.
\end{theorem} 
In the above, a periodic union of stripes of width $h$ is by definition a set which, up to Lebesgue null sets, is of the form $V_i^\perp+\widehat E e_i$ for some $i\in\{1,\dots,d\}$, where $V_i^\perp$ is the $(d-1)$-dimensional subspace orthogonal to $e_i$ and $\widehat E\subset\R$ with $\widehat E=\bigcup_{k=0}^N(2kh+\nu,(2k+1)h+\nu)$ for some $\nu\in\R$ and some $N\in\N$.

Let us now recall the decomposition of the functional $\Fcalt$ obtained in \cite{DKR}[Section 3].
In \cite{DKR}[Proposition 3.1] one gets the following lower bound for the functional $\Fcalt$:
\begin{align}
\label{eq:3.13}
\Fcalt(u)\geq&\frac{1}{L^d} \sum_{i=1}^d\Bigl\{\int_{[0,L)^{d-1}}\Bigl[-\Mi{0}{L}{i}+\Gcal{i}\Bigr]\dx_i^{\perp}+\mathcal I^i_{\tau,L}(u)\Bigr\}\notag\\
&+\frac{1}{L^d}\mathcal W_{\tau,L,\eps}(u),\qquad
\end{align}
where 

\begin{align}
\Mi{s}{t}{i}&:=3\at\int_{[s,t]\cap\{\nabla u(\cdot e_i+x_i^\perp)\neq0\}}|\partial_iu_{x_i^\perp}(\rho)|\|\nabla u(\rho e_i+x_i^\perp)\|_1\d\rho\notag\\
&+\frac{3}{\at}\int_{[s,t]\cap\{\nabla u(\cdot e_i+x_i^\perp)\neq0\}}W(u_{x_i^\perp}(\rho)))\frac{|\partial_iu_{x_i^\perp}(\rho)|}{\|\nabla u(\rho e_i+x_i^\perp)\|_1}\d\rho,
\label{eq:3.2}
\end{align}

\begin{equation}
\label{def:Gcal}
\begin{split}
\Gcal{i}:=&\Mi{0}{L}{i}\int_{\R}|\zeta_i|\widehat K_\tau(\zeta_i)\d\zeta_i - 
\\
&-\int_{\R}\int_0^L|u_{x_i^\perp}(x_i)-u_{x_i^\perp}(x_i+\zeta_i e_i)|^2\widehat K_\tau(\zeta_i)\dx_i\d\zeta_i,
\end{split}
\end{equation}

\begin{equation}
\label{eq:itl}
\mathcal I^i_{\tau,L}(u):=\frac1d\int_{\{\zeta_i>0\}}\int_{[0,L)^d}[(u(x+\zeta_i e_i)-u(x))-(u(x+\zeta)-u(x+\zeta_i^\perp))]^2K_\tau(\zeta)\dx\d\zeta.
\end{equation}
and 
\begin{equation}
\label{def:Wcal}
\mathcal W_{\tau,L,\eps}(u)=\frac{3(C_\tau-1)}{\at}\int_{\{\nabla u=0\}\cap[0,L)^d}W(u(x))\dx,\qquad C_\tau=\int_{\R}|\zeta_i|\widehat K_\tau(\zeta_i)\d\zeta_i.
\end{equation}

In particular, since showing that the minimizers for the \rhs of~\eqref{eq:3.13} are one-dimensional implies that the minimizers for $\Fcal_{\tau,L,\varepsilon}$ are one-dimensional, this allows us to reduce to prove one-dimensionality of the minimizers for the lower bound functional (i.e., the \rhs of \eqref{eq:3.13}).

Given the numerous slicing arguments, it is also convenient to  define the slicing of $\mathcal I^i_{\tau,L}$ as follows

\begin{equation*}
\mathcal I^i_{\tau,L}(u)=\int_{[0,L)^{d-1}}\overline{\mathcal I}^i_{\tau}(u,x_i^\perp,[0,L))\dx_i^\perp,
\end{equation*}
where
\begin{equation}
\label{eq:slicingMathCalI}
\overline{\mathcal I}^i_{\tau}(u,x_i^\perp,[0,L)) :=\frac{1}{d}\int_0^L\int_{\{\zeta_i>0\}}[(u(x+\zeta_i e_i)-u(x))-(u(x+\zeta)-u(x+\zeta_i^\perp))]^2K_\tau(\zeta)\d\zeta\dx_i
\end{equation}
and where $x = x_ie_i + x_i^\perp$.

We also recall the estimate contained in Lemma 4.3 in \cite{DKR}, namely that for all $\rho\in\R$ the following holds:
  \begin{equation}
	\label{eq:partpos0}
	\begin{split}
		|\rho|\Mi{0}{L}{i} & =\int_{0}^L \Mi{s}{s+\rho}{i}\dx_{i}\\ &\geq\int_{0}^L|\omega(u_{x_{i}^{\perp}} (s+  \rho)) - \omega(u_{x_{i}^{\perp}}(s))|\ds,
	\end{split}
\end{equation}
where $\omega:[0,1]\to[0,1]$ is defined by
\begin{equation}
	\label{eq:stimaOmega}
	\omega(t)=\int_0^t6\sqrt{W(s)}\ds=3t^2-2t^3.
\end{equation}
In the first equality of \eqref{eq:partpos0} one uses $[0,L)^d$-periodicity of $u$ and in the second inequality the elementary inequality $a^2+b^2\geq2ab$.

As observed in \cite{DKR}[Remark 4.1], the function  $\omega$ satisfies the following inequality: for $a,b\in[0,1]$ with $a=b+t$, $t>0$

\begin{equation}\label{eq:omegaab}
\frac{\omega(a)-\omega(b)}{|a-b|^2}=\frac{6b(1-b-t)}{t}+3-2t\geq3-2t\geq1.
\end{equation}
and equality in the last inequality holds if and only if $a=1$ and  $b=0$.

In order to measure the $L^1$ distance of the functions $u$ to stripes having boundaries orthogonal to a certain direction we will need the following definition, analogous to the one introduced for sets in \cite{DR}.

\begin{definition}
	\label{def:defDEta}
	For every $\eta$ we denote by $\Acal^{i}_{\eta}$ the family of all sets $F$ such that  
	\begin{enumerate}[(i)]
		\item they are union of stripes with boundaries orthogonal to $e_i$;
		\item their connected components of the boundary are distant at least $\eta$. 
	\end{enumerate}
	We denote by 
	\begin{equation} 
	\label{eq:defDEta}
	\begin{split}
	D^{i}_{\eta}(u,Q) := \inf\Big\{ \frac{1}{|Q|} \int_{Q} |u -\chi_{F}|:\ F\in \Acal^{i}_{\eta} \Big\} \quad\text{and}\quad D_{\eta}(u,Q) = \inf_{i} D^{i}_{\eta}(u,Q).
	\end{split}
	\end{equation} 
\end{definition}

Also in the case of functions, such a distance enjoys some useful properties, that we list below (for a proof see the analogous Remark 7.4 in \cite{DR}).
\begin{remark}
	\label{rmk:lip} \ 
	\begin{enumerate}[(i)]
		\item The map $z\mapsto D_{\eta}(u,Q_{l}(z))$ is Lipschitz, with Lipschitz constant $C_d/l$, where $C_d$ is a constant depending only on the dimension $d$. 
		\item { For every $\nu>0$ there exists ${\sigma}_0= \sigma_0(\nu)$ such that  for every $\sigma \leq \sigma_0 $ whenever $D^{j}_{\eta}(u,Q_{l}(z))\leq \sigma$ and $D^{i}_{\eta}(u,Q_{l}(z))\leq \sigma$ with $i\neq j$ for some $\eta>0$,  it holds 
		\begin{equation}
		\label{eq:gsmstr2}
		\begin{split}
		\min\big\{\|u\|_{L^1(Q_l(z))}, \|1-u\|_{L^1(Q_l(z))} \big\} \leq\nu l^d. 
		\end{split}
		\end{equation}}
		\end{enumerate}
	
	\end{remark}

Finally, we recall from \cite{DKR} (Lemma 5.6) the following

\begin{lemma}
	\label{lemma:int14}
	Let $u$ be such that $D^j_{\eta}(u,Q_l(z))\leq\bar \sigma$ for some $j\neq i$.
	Then, for any $\alpha>0$ and  $|s_0-t_0|\leq\alpha$, if $\bar{\sigma}$ is sufficiently small
	\begin{equation}
		\label{eq:gstr3}
		\begin{split}
			\int_{\{ |\zeta_{j}^\perp| < \alpha \}} \int_{s_{0}-\alpha}^{s_{0}} \int_{t_0-x_j}^{t_0-x_j + \alpha}\!\! \Big\{\frac14- \big[ u(x_j^\perp + \zeta_j^\perp + x_j e_j ) - u(x_j^\perp + \zeta_j^\perp + x_je_j+ \zeta_j e_j)  \big]\Big\}^2 \d\zeta_j\dx_j\d\zeta_j^\perp > \frac{1}{8}\alpha^{d+1}.
		\end{split}
	\end{equation}
\end{lemma}

\section{Decomposition of the functional}\label{sec:dec}
 The aim of  this section is to prove the localization and decomposition of the functional $\Fcalt$ contained in  Lemma \ref{lemma:dec}. Moreover, at the end of this section, we can state the Local Rigidity Proposition \ref{lemma:local_rigidity}. The decomposition of Lemma \ref{lemma:dec} is new and is able to single out the different contributions to the energies of intervals of one-dimensional slices. In order to find it we use first $L$-periodicity of $u$ and then a disintegration argument.  
 
In order to state Lemma \ref{lemma:dec}, one needs some preliminary definitions.

For $a\in[0,L)$, $b\in\R$ define
\begin{equation}\label{eq:Omegaab}
\Omega(a,b)=\left\{\begin{aligned}
(s,\rho)\in [0,L]\times\R:[s,s+\rho]\supset[a,b]\text{ or }[s+\rho,s]\supset [a,b]\text{ if $a\leq b$}\\
(s,\rho)\in [0,L]\times\R:[s,s+\rho]\supset[b,a]\text{ or }[s+\rho,s]\supset [b,a]\text{ if $b\leq a$}
\end{aligned}\right\}
\end{equation}
and for all  $\rho\in\R$ let
\begin{equation}\label{eq:gsrho}
G(\rho)=|\rho|\min\{|\rho|,L\}.
\end{equation}
Notice that, for any $s\in[0,L)$, $\rho\in\R$
\begin{equation}\label{eq:gom}
\int_{s}^{s+L}\int_{\R}\chi_{\{(a,b):(s,\rho)\in\Omega(a,b)\}}(a,b)\db\da=G(\rho).
\end{equation}
Then, for any $i\in\{1,\dots,d\}$, $x_{i}^\perp\in[0,L)^{d-1}$ and for any interval $I\subset[0,L)$ define
\begin{align}
R_{i,\tau,\eps}(u,&x_i^\perp, I)=-\overline{\mathcal M}^i_{\at}(u,x_i^\perp, I)\notag\\
&+\int_I\int_{\R}\iint_{(s,\rho)\in\Omega(a,b)}G^{-1}(\rho)\Bigl(\Mi{s}{s+\rho}{i}-(u(s)-u(s+\rho))^2\Bigr)\widehat K_\tau(\rho)\Bigr]\ds\d\rho\db\da,\label{eq:ritau}
\end{align}

\begin{align}
V_{i,\tau}(u,x_i^\perp,I)&=\frac{1}{2d}\int_I\int_{\R}\iint_{(s,\rho)\in\Omega(a,b)}\int_{\R^{d-1}}G^{-1}(\rho)f_u(x_i^\perp,s,y_i^\perp,s+\rho)K_\tau(\rho e_i+(y_i^\perp-x_{i}^\perp))\dy_i^\perp\ds\d\rho\db\da\label{eq:vitau}
\end{align}
with
\begin{align}\label{eq:fu}
	f_u(x_i^\perp,s,y_i^\perp,s+\rho)=[(u(x_i^\perp+(s+\rho)e_i)-u(x_i^\perp+se_i))-(u(x_i^\perp+y_i^\perp+(s+\rho)e_i)-u(x_i^\perp+y_i^\perp+se_i))]^2,
\end{align}
and 
\begin{align}
W_{i,\tau}(u,x)=\frac{1}{2d}\int_{\R^d}f_u(x_i^\perp,x_i,y_i^\perp,y_i)K_\tau(y-x)\dy.
\end{align}

We can now state the main result of this section 
\begin{lemma}\label{lemma:dec}
	One has that
	\begin{equation}
		\Fcalt(u)\geq\frac{1}{L^d}\sum_{i=1}^d\int_{[0,L)^d}\bar F_{i,\tau,\eps}(u,Q_l(z))\dz,
	\end{equation}
where 
	\begin{align}
	\bar{F}_{i,\tau,\eps}(u,Q_{l}(z))&=\frac{1}{l^d}\int_{Q_l^\perp(z_i^\perp)}R_{i,\tau,\eps}(u,x_i^\perp,(z_i-l/2,z_i+l/2))\dx_i^\perp\notag\\
	&+\frac{1}{l^d}\int_{Q_l^\perp(z_i^\perp)}V_{i,\tau}(u,x_i^\perp,(z_i-l/2,z_i+l/2))\dx_i^\perp\notag\\
	&+\frac{1}{l^d}\int_{Q_l(z)}W_{i,\tau}(u,x)\dx\notag\\
	&+\frac{3(C_\tau-1)}{dl^d}\int_{\{\nabla u=0\}\cap Q_l(z)}\frac{W(u(x))}{\at}\dx.\label{eq:3.9}
	\end{align}
\end{lemma}

\begin{proof}
	\textbf{Step 1}
	
Recall the lower bound \eqref{eq:3.13} obtained in \cite{DKR}[Proposition 3.1], namely 
\begin{align}
\label{eq:3.13bis}
\Fcalt(u)\geq&\frac{1}{L^d} \sum_{i=1}^d\Bigl\{\int_{[0,L)^{d-1}}\Bigl[-\Mi{0}{L}{i}+\Gcal{i}\Bigr]\dx_i^{\perp}+\mathcal I^i_{\tau,L}(u)\Bigr\}\notag\\
&+\frac{1}{L^d}\mathcal W_{\tau,L,\eps}(u),\qquad
\end{align}
where $\overline{\mathcal M}^i_{\at}$, $\overline{\mathcal G}^i_{\at,\tau}$, $\mathcal I^i_{\tau,L}$ and $\mathcal W_{\tau,L,\eps}$ are defined in \eqref{eq:3.2}-\eqref{def:Wcal}.

Using the periodicity of $u$ w.r.t. $[0,L)^d$, as recalled in \eqref{eq:partpos0} one has that
\begin{equation}
|\rho|\Mi{0}{L}{i}=\int_0^L\Mi{s}{s+\rho}{i}\ds,
\end{equation} 
where by convention from now onwards
\begin{equation}
[s,s+\rho)=\left\{\begin{aligned}
&[s,s+\rho) && &\text{if $\rho>0$},\\
&[s+\rho,s) && &\text{if $\rho<0$}.
\end{aligned}\right.
\end{equation}
Therefore 
\begin{equation}\label{eq:gper}
\overline{\mathcal{G}}^i_{\at,\tau}(u,x_i^\perp,[0,L))=\int_0^L\int_{\R}\Bigl[\Mi{s}{s+\rho}{i}-(u(s)-u(s+\rho))^2\Bigr]\widehat K_\tau(\rho)\d\rho\ds.
\end{equation}
On the other hand, recall that
\begin{equation}\label{eq:islice}
\mathcal I^i_{\tau,L}(u)=\int_{[0,L)^{d-1}}\overline{\mathcal I}^i_{\tau}(u,x_i^\perp,[0,L))\dx_i^\perp,
\end{equation}
where
\begin{align}
\label{eq:slicingMathCalIbis}
\overline{\mathcal I}^i_{\tau}(u,x_i^\perp,[0,L)) &:=\frac{1}{d}\int_0^L\int_{\{\rho>0\}}\int_{\R^{d-1}}f_u(x_i^\perp,s,y_i^\perp,s+\rho)K_\tau(\rho e_i+(y_i^\perp-x_{i}^\perp))\dy_i^\perp\d\rho\ds
\end{align}
and $f_u$ as in \eqref{eq:fu}.

Hence, decomposing half of the term  $\mathcal{I}^i_{\tau,L}$ as in \eqref{eq:islice}-\eqref{eq:slicingMathCalIbis} and leaving half of it as it is, it follows immediately that
\begin{align}
\mathcal I^i_{\tau,L}(u)&=\int_{[0,L)^{d-1}}\Bigl[\frac{1}{2d}\int_0^L\int_{\{\rho>0\}}\int_{\R^{d-1}}f_u(x_i^\perp,s,y_i^\perp,s+\rho)K_\tau(\rho e_i+(y_i^\perp-x_{i}^\perp))\dy_i^\perp\d\rho\ds\Bigr]\dx_i^\perp\notag\\
&+\frac{1}{2d}\int_{[0,L)^d}W_{i,\tau}(u,x)\dx.\label{eq:Idec}
\end{align}

Hence, by \eqref{eq:3.13bis}, \eqref{eq:gper} and  \eqref{eq:Idec} one has that  
\begin{align}
\Fcalt(u)\geq&\frac{1}{L^d} \sum_{i=1}^d\int_{[0,L)^{d-1}}\Bigl\{-\Mi{0}{L}{i}\notag\\
&+\int_0^L\int_{\R}\Bigl[\Mi{s}{s+\rho}{i}-(u(s)-u(s+\rho))^2\Bigr]\widehat K_\tau(\rho)\d\rho\ds\notag\\
&+\frac{1}{2d}\int_0^L\int_{\{\rho>0\}}\int_{\R^{d-1}}f_u(x_i^\perp,s,y_i^\perp,s+\rho)K_\tau(\rho e_i+(y_i^\perp-x_{i}^\perp))\dy_i^\perp\d\rho\ds\Bigr\}\dx_i^\perp\notag\\
&+\frac{1}{2dL^d}\sum_{i=1}^d\int_{[0,L)^d}W_{i,\tau}(u,x)\dx\notag\\
&+\frac{1}{L^d}\mathcal W_{\tau,L,\eps}(u).\label{eq:lowboundL}
\end{align}

\textbf{Step 2} By \eqref{eq:gsrho}, \eqref{eq:gom} and Fubini Theorem, for any function $A(s,\rho)$ with $A(\cdot,\rho)$ $L$-periodic one has that
\begin{align}\label{eq:disint}
\int_0^L\int_{\R}A(s,\rho)\ds\d\rho&=\int_0^L\int_{\R}A(s,\rho)G^{-1}(\rho)G(\rho)\d\rho\ds\notag\\
&=\int_0^L\int_{\R}A(s,\rho)G^{-1}(\rho)\int_s^{s+L}\int_{\R}\chi_{\{(a,b):(s,\rho)\in\Omega(a,b)\}}(a,b)\db\da\d\rho\ds\notag\\
&=\int_0^L\int_{\R}A(s,\rho)G^{-1}(\rho)\int_0^{L}\int_{\R}\chi_{\{(a,b):(0,\rho)\in\Omega(a-s,b-s)\}}(a,b)\db\da\d\rho\ds\notag\\
&=\int_0^L\int_{\R}\int_0^{L}\int_{\R}A(s,\rho)G^{-1}(\rho)\chi_{\{(a,b):(0,\rho)\in\Omega(a-s,b-s)\}}(a,b)\d\rho\ds\db\da\notag\\
&\int_0^L\int_{\R}\iint_{(s,\rho)\in\Omega(a,b)}A(s,\rho)G^{-1}(\rho)\d\rho\ds\db\da.
\end{align}

Applying  \eqref{eq:disint} to the second and the third term of the r.h.s. of  \eqref{eq:lowboundL} and recalling the definitions of $R_{i,\tau,\eps}$ and $V_{i,\tau}$ given in \eqref{eq:ritau} and \eqref{eq:vitau}, one obtains 

\begin{align}
\Fcal_{\tau,L,\eps}(u)\geq&\frac{1}{L^d} \sum_{i=1}^d\Bigl\{\int_{[0,L)^{d-1}}\Bigl\{R_{i,\tau,\eps}(u,x_i^\perp,[0,L))+V_{i,\tau}(u,x_i^\perp,[0,L))\Bigr\}\dx_i^\perp\notag\\
&+\frac{1}{2dL^d}\int_{[0,L)^d}W_{i,\tau}(u,x)\dx\Bigr\}\notag\\
&+\frac{1}{L^d}\mathcal W_{\tau,L,\eps}(u).
\end{align}
 
 \textbf{Step 3}
By periodicity of $u$ w.r.t. $[0,L)^d$, as in \cite{DR}, one has that 
\[
\Fcalt(u)\geq\frac{1}{L^d}\int_{[0,L)^d}\bar F_{\tau,\eps}(u,Q_l(z))\dz
\]
with
\[
\bar F_{\tau,\eps}(u,Q_l(z)):=\sum_{i=1}^d\bar F_{i,\tau,\eps}(u,Q_l(z))
\]
and 
\begin{align}
\bar F_{i,\tau,\eps}(u,Q_l(z))&=\frac{1}{l^d}\int_{Q_l^\perp(z_i^\perp)}\Bigl\{R_{i,\tau,\eps}(u,x_i^\perp,Q_l^i(z_i))+V_{i,\tau}(u,x_i^\perp,Q_l^i(z_i))\Bigr\}dx_i^\perp\notag\\
&+\frac{1}{2dl^d}\int_{Q_l(z)}W_{i,\tau}(u,x)\dx\notag\\
&+\frac{3(C_\tau-1)}{dl^d}\int_{\{\nabla u=0\}\cap Q_l(z)}\frac{W(u(x))}{\at}\dx,\label{eq:fbari}
\end{align}
where 
\begin{equation}
	C_\tau=\int_{\R}|\zeta_i|\widehat K_\tau(\zeta_i)\d\zeta_i.
\end{equation}

From Theorem \ref{thm:gammaconv}, Theorem \ref{Thm:DR} and Definition \ref{def:defDEta} one has the following 
\begin{proposition}[Local Rigidity] 
	\label{lemma:local_rigidity}
	For every $M > 1,l,\sigma > 0$, there exist ${\hat\tau},{\hat\eta},\hat\eps >0$ 
	such that whenever $\tau< \hat{\tau}$, $\eps<\hat\eps$  and $\bar F_{\tau,\eps}(u,Q_{l}(z)) < M$ for some $z\in [0,L)^d$ and $u\in W^{1,2}_{\mathrm{loc}}(\R^d;[0,1])$ $[0,L)^d$-periodic, with $L>l$, then it holds $D_{\eta}(u,Q_{l}(z))\leq\sigma$ for every $\eta < \hat{\eta}$. Moreover $\hat{\eta}$ can be chosen independent on $\sigma$.  Notice that ${\hat\tau}$, $\hat\eps$ and ${\hat\eta}$ are independent of $L$.
\end{proposition} 

\end{proof}

\section{Preliminary lemmas}
This section contains all the preliminary lemmas which are needed to prove Theorem \ref{thm:main}. Such lemmas  provide a series of estimates which show an increase of the localized energy when $u$ deviates in different ways from being one-dimensional.  This section, together with Section \ref{sec:dec}, contains the main novelties w.r.t. the study of the corresponding sharp interface problem, while Section \ref{sec:mainthm} recalls the main underlying strategy and how the lemmas of this section enter in the proof of Theorem \ref{thm:main} and is very similar to \cite{DR}[Section 7].  

Let us start recalling from \cite{DKR} the following facts:
\begin{lemma}\label{lemma:pos0}
	For all $i\in\{1,\dots,d\}$, $x_i^\perp\in[0,L)^{d-1}$, $s\in[0,L)$, $\rho\in\R$
	\begin{equation}\label{eq:pos1}
	\Mi{s}{s+\rho}{i}-(u(s)-u(s+\rho))^2\geq0
	\end{equation}
	and if $|u(s)-u(s+\rho)|\leq1-\delta$
	\begin{equation}\label{eq:pos2}
	\frac{1}{1+2\delta}\Mi{s}{s+\rho}{i}-(u(s)-u(s+\rho))^2\geq0.
	\end{equation}
\end{lemma}
The proof, follows immediately from the fact that $\Mi{s}{s+\rho}{i}\geq|\omega(u(s))-\omega(u(s+\rho))|$, where $\omega(t)=3t^2-2t^3$ is the optimal energy function for the Modica-Mortola term and on the fact that $|\omega(a)-\omega(b)|\geq (3-2|a-b|)|a-b|^2$ (see \eqref{eq:partpos0} and \eqref{eq:stimaOmega}).

In the following lemma, penalization of functions $u$ whose Modica-Mortola term is large on a small interval is shown. In particular, since the minimal energy is negative whenever $\tau$ and $\eps$ are sufficiently small and since $V_{i,\tau}$ and $W_{i,\tau}$ in \eqref{eq:3.9} are nonnegative, whenever we show that under some conditions also the third term contributing to the energy, i.e. $R_{i,\tau,\eps}$, is positive, then the corresponding configuration is surely not optimal. 

\begin{lemma}\label{lemma:rpos1}
	Let $\Upsilon>1$.  Then there exists $C>0$,  $\eta_0=\eta_0(\Upsilon)$ and $\tau_0>0$ with ${\tau_0}^{1/\beta}\leq\eta_0$ such that for every $\tau\leq{\tau_0}$, $k\geq1$ and for all $x_i^\perp\in[0,L)^d$, $\bar a\in[0,L)$ such that $\Mi{\bar a-\eta_0}{\bar a+\eta_0}{i}\geq k\Upsilon$, then $R_{i,\tau,\eps}(u,x_i^\perp,(\bar a-\eta_0/2,\bar a+\eta_0/2))\geq k\Upsilon>0$.
\end{lemma}

\begin{proof}
	First of all, observe that for any $\eta_0$ such that $\Mi{\bar a-\eta_0}{\bar a+\eta_0}{i}\geq k\Upsilon$ and for all  $s,s+\rho\in\Omega(\bar a-\eta_0,\bar a+\eta_0) $
	\begin{align}
	\Mi{s}{s+\rho}{i}-(u(s)-u(s+\rho))^2&=\frac{\Upsilon-1}{\Upsilon}	\Mi{s}{s+\rho}{i}\notag\\
	&+\frac{1}{\Upsilon}	\Mi{s}{s+\rho}{i}-(u(s)-u(s+\rho))^2\notag\\
	&\geq \frac{\Upsilon-1}{\Upsilon}	\Mi{s}{s+\rho}{i},\label{eq:upsm}
	\end{align}
	where in the last inequality have used that $\Mi{s}{s+\rho}{i}\geq\Mi{\bar a-\eta_0}{\bar a+\eta_0}{i}\geq\Upsilon$ and $|u(s)-u(s+\rho)|\leq1$. 
	Recalling \eqref{eq:pos1} and then using Fubini Theorem, one  has that
	\begin{align}
	R_{i,\tau,\eps}&(u,x_i^\perp,(\bar a-\eta_0/2,\bar a+\eta_0/2))=-\Mi{\bar a-\eta_0/2}{\bar a+\eta_0/2}{i}\notag\\
	&+\int_{\bar a-\eta_0/2}^{\bar a+\eta_0/2}\int_{\R}\iint_{(s,\rho)\in\Omega(a,b)}G^{-1}(\rho)\Bigl(\Mi{s}{s+\rho}{i}-(u(s)-u(s+\rho))^2\Bigr)\widehat K_\tau(\rho)\Bigr]\ds\d\rho\db\da\notag\\
	&\geq-\Mi{\bar a-\eta_0/2}{\bar a+\eta_0/2}{i}+\int_{\bar a-\eta_0/2}^{\bar a+\eta_0/2}\int_{\R}\iint_{(s,\rho)\in\Omega(a,b)\cap \Omega (\bar a-\eta_0,\bar a+\eta_0)}G^{-1}(\rho)\widehat K_\tau(\rho)\cdot\notag\\
	&\cdot \Bigl(\Mi{s}{s+\rho}{i}-(u(s)-u(s+\rho))^2\Bigr)\ds\d\rho\db\da\notag\\
	&\geq-\Mi{\bar a-\eta_0/2}{\bar a+\eta_0/2}{i}+\iint_{(s,\rho)\in\Omega (\bar a-\eta_0,\bar a+\eta_0),\{|\rho|\leq4\eta_0\}}G^{-1}(\rho)\widehat K_\tau(\rho)\cdot\notag\\
	&\cdot \Bigl(\Mi{s}{s+\rho}{i}-(u(s)-u(s+\rho))^2\Bigr)\int_{\bar a-\eta_0/2}^{\bar a+\eta_0/2}\int_{\R}\chi_{\{(a,b):(s,\rho)\in\Omega(a,b)\}}(a,b)\da\db\ds\d\rho.\label{eq:Rineq1}
	\end{align}
	Now observe that by direct computation, whenever $(s,\rho)\in\Omega(\bar a-\eta_0,\bar a+\eta_0)$,
	\[
	\int_{\bar a-\eta_0/2}^{\bar a+\eta_0/2}\int_{\R}\chi_{\{(a,b):(s,\rho)\in\Omega(a,b)\}}(a,b)\da\db=|\rho|\eta_0,
	\] 
	hence (recalling the definition of $G(\rho)$ in \eqref{eq:gsrho} and the fact that $|\rho|\leq4\eta_0$ with $4\eta_0\ll L$)
	\begin{equation}\label{eq:g-1ineq}
	G^{-1}(\rho)	\int_{\bar a-\eta_0/2}^{\bar a+\eta_0/2}\int_{\R}\chi_{\{(a,b):(s,\rho)\in\Omega(a,b)\}}(a,b)\da\db=\frac{\eta_0}{|\rho|}\geq\frac14.
	\end{equation}
	Inserting \eqref{eq:g-1ineq} in \eqref{eq:Rineq1}, one obtains
	\begin{align*}
		R_{i,\tau,\eps}&(u,x_i^\perp,(\bar a-\eta_0/2,\bar a+\eta_0/2))\geq-\Mi{\bar a-\eta_0/2}{\bar a+\eta_0/2}{i}\notag\\
		&+\frac14\iint_{(s,\rho)\in\Omega (\bar a-\eta_0,\bar a+\eta_0),\{|\rho|\leq4\eta_0\}}\Bigl(\Mi{s}{s+\rho}{i}-(u(s)-u(s+\rho))^2\Bigr)\widehat K_\tau(\rho)\ds\d\rho.
	\end{align*}
	Now since $(s,\rho)\in\Omega (\bar a-\eta_0,\bar a+\eta_0)$, we can use \eqref{eq:upsm} and get
	\begin{align}
		R_{i,\tau,\eps}&(u,x_i^\perp,(\bar a-\eta_0/2,\bar a+\eta_0/2))\geq-\Mi{\bar a-\eta_0/2}{\bar a+\eta_0/2}{i}\notag\\
		&+\frac{\Upsilon-1}{4\Upsilon}\Mi{\bar a-\eta_0}{\bar a+\eta_0}{i}\iint_{(s,\rho)\in\Omega (\bar a-\eta_0,\bar a+\eta_0),\{|\rho|\leq4\eta_0\}}\frac{1}{(|\rho|+\tau^{1/\beta})^q}\d\rho\ds\notag\\
		&\geq \Mi{\bar a-\eta_0}{\bar a+\eta_0}{i}\Bigl(-1+\frac{\Upsilon-1}{4\Upsilon}2\eta_0\int_{3\eta_0\leq|\rho|\leq4\eta_0}\frac{1}{(|\rho|+\tau^{1/\beta})^q}\d\rho\Bigr)\notag\\
		&\geq k\Bigl(-\Upsilon+\frac{\bar C(\Upsilon-1)}{\eta_0^{q-2}}\Bigr)\notag\\
		&\geq k\Upsilon\label{eq:estlemmapos1}
	\end{align}
	 where the last two inequalities hold provided $\tau^{1/\beta}\leq\eta_0$ and $\eta_0$ is sufficiently small depending on $\Upsilon$.
	
\end{proof}
As a consequence, the following holds. 
\begin{corollary}\label{cor:pos0}
	Let $\Upsilon>1$ and $\eta_0,\tau_0$ as in Lemma \ref{lemma:rpos1}. Then, for all intervals  $I\subset[0,L)$ such that $R_{i,\tau,\eps}(u,x_i^\perp,I)<0$ and for all $\tau\leq\tau_0$, it holds 
	\begin{equation}\label{eq:mineq}
	\overline{\mathcal M}^i_{\at}(u,x_i^\perp,I)\leq{2\Upsilon}\max\Bigl\{\frac{|I|}{\eta_0},1\Bigr\}.
	\end{equation}
	
\end{corollary}

\begin{proof}
	Let $I=(a,b)$ and let $a=t_0<t_1<\dots<t_N=b$ be a partition of $I$ into intervals $[t_j,t_{j+1}]$ such that the following holds:
	\begin{equation*}
	t_{j+1}=\inf\bigl\{t\in I, t\geq t_j+\eta_0/2: \overline {\mathcal M}^i_{\at}(u,x_i^\perp,[t_j,t))\geq\Upsilon\bigr\}.
	\end{equation*} W.l.o.g., assume that $|I|\geq\eta_0$. Otherwise by Lemma \ref{lemma:rpos1} one has that  \eqref{eq:mineq} holds. Let $s\geq1$ and let us define the following sets:
	\begin{align*}
	A_s(\eta_0)&:=\{j\in\{0,N-1\}:\,|t_{j+1}-t_j|\leq\eta_0\text{ and } \overline {\mathcal M}^i_{\at}(u,x_i^\perp,[t_j,t_{j+1}))= s\Upsilon\}\notag\\
		A_0(\eta_0)&:=\{j\in\{0,N-2\}:\,|t_{j+1}-t_j|\geq\eta_0\}.
	\end{align*}
	By definition, if $j\in A_0$ then $\overline {\mathcal M}^i_{\at}(u,x_i^\perp,[t_j,t_{j+1}))=\Upsilon$. 
	One has that
	\begin{align}\label{eq:mineq2}
		\overline{\mathcal M}^i_{\at}(u,x_i^\perp,I)\leq\sum_{j\in A_0(\eta_0)}\Upsilon+\sum_{s\geq1}\sum_{j\in A_s(\eta_0)}s\Upsilon=\Upsilon \#A_0(\eta_0)+\sum_{s\geq1}s\Upsilon \#A_s(\eta_0).
	\end{align}
Assume now that \eqref{eq:mineq} does not hold. 
Hence, since  $\#A_0(\eta_0)\leq\frac{|I|}{\eta_0}$ and using \eqref{eq:mineq2}, one has that
\begin{align}\label{eq:4.12}
\Upsilon\#A_0(\eta_0)\leq\frac{\Upsilon|I|}{\eta_0}\leq\frac12\Upsilon\#A_0(\eta_0)+\frac12\sum_{s\geq1}s\Upsilon \#A_s(\eta_0). 
\end{align}

	On the other hand, by Lemma \ref{lemma:rpos1} and \eqref{eq:4.12}
	\begin{align*}
	R_{i,\tau,\eps}(u,x_i^\perp,I)&=\sum_{j\in A_0(\eta_0)}	R_{i,\tau,\eps}(u,x_i^\perp,[t_j,t_{j+1}))+\sum_{s\geq1}\sum_{j\in A_s(\eta_0)}	R_{i,\tau,\eps}(u,x_i^\perp,[t_j,t_{j+1}))\notag\\
	&\geq-\Upsilon \#A_0(\eta_0)+\Upsilon\sum_{s\geq1}s\#A_s(\eta_0)\notag\\
	&\geq0,
	\end{align*}
	thus reaching a contradiction.

\end{proof}

The following lemma shows that non optimal configurations are also those for which oscillations of amplitude close to $1$ happen at a scale larger than $\tau^{1/\beta}$. 

\begin{lemma}\label{lemma:rpos2}
	Let $\delta_0\geq\tau^{1/\beta}$, $\delta>0$. Then, there exists $\tau_1>0$ such that for all $\tau\leq\tau_1$, for any $x_i^\perp\in[0,L)^{d-1}$, $\bar a\in[0,L)$ and $c>0$, if for any $s,s+\rho\in[\bar a-c-\delta_0/2,\bar a+c+\delta_0/2]$ with $|\rho|\leq\delta_0$ one has that
	\[
	|u(s)-u(s+\rho)|\leq1-\delta,
	\]  
	then $R_{i,\tau,\eps}(u,x_i^\perp,(\bar a-c,\bar a+c))>0$.
\end{lemma}

\begin{proof}
	Recalling \eqref{eq:pos1} and \eqref{eq:pos2} and then using Fubini Theorem, one has that
	\begin{align}
	R_{i,\tau,\eps}&(u,x_i^\perp,(\bar a-c,\bar a+c))\geq-\Mi{\bar a-c}{\bar a+c}{i}\notag\\
	&+\int_{\bar a-c}^{\bar a+c}\int_{\R}\iint_{(s,\rho)\in\Omega(a,b)\cap \{s,s+\rho\in[\bar a-c-\delta_0/2,\bar a+c+\delta_0/2],\,\frac12\tau^{1/\beta}\leq|\rho|\leq\delta_0\}}G^{-1}(\rho)\widehat K_\tau(\rho)\cdot\notag\\
	&\cdot \frac{2\delta}{1+2\delta}\Mi{s}{s+\rho}{i}\ds\d\rho\db\da\notag\\
	&\geq-\Mi{\bar a-c}{\bar a+c}{i}+\iint_{ \{s,s+\rho\in[\bar a-c-\delta_0/2,\bar a+c+\delta_0/2],\,\frac12\tau^{1/\beta}\leq|\rho|\leq\delta_0\}}G^{-1}(\rho)\widehat K_\tau(\rho)\cdot\notag\\
	&\cdot \frac{2\delta}{1+2\delta}\Mi{s}{s+\rho}{i}\int_{\bar a-c}^{\bar a+c}\int_{\R}\chi_{\{(a,b):(s,\rho)\in\Omega(a,b)\}}(a,b)\da\db\ds\d\rho.\label{eq:Rineq2}
	\end{align}
	
	Then observe that 
	\[
	\int_{\bar a-c}^{\bar a+c}\int_{\R}\chi_{\{(a,b):(s,\rho)\in\Omega(a,b)\}}(a,b)\da\db\geq{|[s,s+\rho]\cap[\bar a-c,\bar a+c]||\rho|}
	\]
	and thus, recalling \eqref{eq:gsrho} and assuming w.l.o.g. that $\delta_0\ll L$, one has that
	\begin{equation*}
	G^{-1}(\rho)	\int_{\bar a-c}^{\bar a+c}\int_{\R}\chi_{\{(a,b):(s,\rho)\in\Omega(a,b)\}}(a,b)\da\db\geq \frac{|[s,s+\rho]\cap[\bar a-c,\bar a+c]|}{|\rho|}.
	\end{equation*}
In particular, when $s,s+\rho\in[\bar a-c-\frac{1}{4}\tau^{1/\beta},\bar a+c+\frac14\tau^{1/\beta}]$ and $\frac12\tau^{1/\beta}\leq|\rho|\leq\delta_0$,
one has that
\begin{equation}\label{eq:g-1ineq3}
		G^{-1}(\rho)	\int_{\bar a-c}^{\bar a+c}\int_{\R}\chi_{\{(a,b):(s,\rho)\in\Omega(a,b)\}}(a,b)\da\db\geq \frac{|[s,s+\rho]\cap[\bar a-c,\bar a+c]|}{|\rho|}\geq\frac{1}{2}.
\end{equation}
 Restricting further the domain of integration  in \eqref{eq:Rineq2} to the pairs $s,s+\rho\in[\bar a-c-\frac14\tau^{1/\beta}, \bar a+c+\frac14\tau^{1/\beta}]$ 	and inserting \eqref{eq:g-1ineq3}, we obtain
	\begin{align}
		R_{i,\tau,\eps}&(u,x_i^\perp,(\bar a-c,\bar a+c))\geq-\Mi{\bar a-c}{\bar a+c}{i}\notag\\
		&+\frac{\delta}{1+2\delta}\iint_{ \{s,s+\rho\in[\bar a-c-\frac14\tau^{1/\beta},\bar a+c+\frac14\tau^{1/\beta}],\,\frac12\tau^{1/\beta}\leq|\rho|\leq\delta_0\}}\Mi{s}{s+\rho}{i}\widehat K_\tau(\rho)\ds\d\rho.\label{eq:4.11}
		\end{align}
	Then observe that, whenever $\frac12\tau^{1/\beta}\leq|\rho|\leq\tau^{1/\beta}$ 
	\begin{equation}\label{eq:rhoper}
	\int_{[\bar a-c-\frac14\tau^{1/\beta}, \bar a+c+\frac14\tau^{1/\beta}]}\Mi{s}{s+\rho}{i}\geq\frac{|\rho|}{4}\Mi{\bar a-c}{\bar a+c}{i}.
	\end{equation}
Hence, inserting \eqref{eq:rhoper} in \eqref{eq:4.11} one gets 
\begin{align*}
		R_{i,\tau,\eps}&(u,x_i^\perp,(\bar a-c,\bar a+c))\geq\Mi{\bar a-c}{\bar a+c}{i}\Bigl(-1+\frac{\delta}{1+2\delta}\int_{\frac12\tau^{1/\beta}\leq|\rho|\leq\tau^{1/\beta}}\frac{|\rho|}{4(|\rho|+\tau^{1/\beta})^q}\d\rho\Bigr)
	\end{align*}
	which is positive provided $\tau$ is sufficiently small depending on $\delta$. 
\end{proof}

\begin{corollary}\label{cor:pos}
	Let $1<\Upsilon\leq{\frac{17}{16}}$, $\delta>0$ and let $\eta_0$ as in Lemma \ref{lemma:rpos1}, $\delta_0$ such that  $\tau^{1/\beta}\leq\delta_0\ll\eta_0$ as in Lemma \ref{lemma:rpos2} and  $\tau_2\leq\min\{\tau_0,\tau_1\}$. Then, for all $\tau\leq\tau_2$, for all $x_i^\perp\in[0,L)^{d-1}$, $\bar a\in[0,L)$, whenever $R_{i,\tau,\eps}(u,x_i^\perp,(\bar a-\eta_0/2,\bar a+\eta_0/2))<0$ then
	\begin{enumerate}
		\item $\Mi{\bar a-\eta_0}{\bar a+\eta_0}{i}\leq\Upsilon$\\
		\item There exist $s_0<t_0\in[\bar a-\eta_0/2-\delta_0/2, \bar a+\eta_0/2+\delta_0/2]$ with $|t_0-s_0|\leq\delta_0$ such that $|u_{x_i^\perp}(s_0)-u_{x_i^\perp}(t_0)|\geq(1-\delta)$.
	\end{enumerate}
In particular, 
\begin{equation}\label{eq:Rgamma}R_{i,\tau,\eps}(u,x_i^\perp,(\bar a-\eta_0/2,\bar a+\eta_0/2))\geq-\Upsilon
\end{equation} and 
\begin{align}
&\forall\,t\in[t_0,\bar a+\eta_0],\quad|u_{x_i^\perp}(t)-u_{x_i^\perp}(t_0)|\leq\frac{1}{4}+\sqrt{2\delta}\label{eq:close1}\\
&\forall\,s\in[\bar a-\eta_0,s_0],\quad|u_{x_i^\perp}(s)-u_{x_i^\perp}(s_0)|\leq\frac{1}{4}+\sqrt{2\delta}.\label{eq:close2}
\end{align}

\end{corollary}

\begin{proof}
	The only estimates which are not restatements of Lemma \ref{lemma:rpos1} and Lemma \ref{lemma:rpos2} are \eqref{eq:close1} and \eqref{eq:close2}. 
	
	To this aim, notice that by \eqref{eq:pos2}
	   \begin{align*}
	\Mi{t_0}{\bar a+\eta_0}{i}+\Mi{\bar a-\eta_0}{s_0}{i}&=\Mi{\bar a-\eta_0}{\bar a+\eta_0}{i}-\Mi{s_0}{t_0}{i}\notag\\
	&\leq\frac{17}{16}-(u_{x_i^\perp}(s_0)-u_{x_i^\perp}(t_0))^2\notag\\
	&\leq\frac{1}{16}+2\delta.
	\end{align*}
	Thus, by \eqref{eq:pos1},
	\begin{align}
	|u_{x_i^\perp}(s)-u_{x_i^\perp}(s_0)|&\leq\sqrt{\Mi{s}{s_0}{i}}\leq\frac14+\sqrt{2\delta}\quad\forall\,s\in[\bar a-\eta_0,s_0],\label{eq:381}\\
	|u_{x_i^\perp}(t)-u_{x_i^\perp}(t_0)|&\leq\sqrt{\Mi{t_0}{t}{i}}\leq\frac14+\sqrt{2\delta}\quad\forall\,t\in[t_0,\bar a+\eta_0].\label{eq:581}
	\end{align}
In particular, the upper bound $\Upsilon\leq\frac{17}{16}$ enters only in the proof of \eqref{eq:close1} and \eqref{eq:close2}.

\end{proof}

The following lemma (diffuse interface counterpart of Lemma 7.7 in \cite{DR}) contains a localized version of the one-dimensional  minimization of the slices of the functional. In comparison with the sharp interface problem, where a periodic extension argument is also used, one has additionally to take care of the growth of the Modica-Mortola term of the periodic extension. At this purpose we will use Corollary \ref{cor:pos0}.

\begin{lemma}\label{lemma:1dopt}
	There exists $C_0>0$ with the following property. Let $u\in W^{1,2}_{\mathrm{loc}}(\R^d;[0,1])$ $[0,L)^d$-periodic, $x_i^\perp\in[0,L)^{d-1}$ and $I\subset[0,L)$ be an open interval. Then, for all $\tau\leq\tau_2$ with $\tau_2$ as in Corollary \ref{cor:pos}  it holds
	\begin{equation}\label{eq:stima1dper}
	R_{i,\tau,\eps}(u,x_i^\perp,I)\geq C^*_{\tau,\eps}|I|-C_0,
	\end{equation}
where $C^*_{\tau,\eps}$ was defined in \eqref{eq:cstartau}.
\end{lemma}

\begin{proof}
	Let $I=[x_1,x_2]$. Let $1<\Upsilon\leq\frac{17}{16}$ and let $\eta_0$ be as in Lemma \ref{lemma:rpos1} and Corollary \ref{cor:pos}. W.l.o.g., we can assume that
	\begin{equation}\label{eq:min0}
	 R_{i,\tau,\eps}(u,x_i^\perp,[x_1,x))<0\text{ and } R_{i,\tau,\eps}(u,x_i^\perp,(x,x_2])<0\quad\text{ for all }x\in (x_1,x_2).
	\end{equation} Indeed, if the contrary holds, setting $\bar I=I\setminus[x_1,x)$ or $\bar I=I\setminus (x,x_2]$ whenever the result of the lemma holds for $\bar I$ then it holds also for $I$, since it holds  
	\[
	R_{i,\tau,\eps}(u,x_i^\perp,I)\geq R_{i,\tau,\eps}(u,x_i^\perp,\bar I)\geq C^*_{\tau,\eps}|\bar I|-C_0\geq C^*_{\tau,\eps}|I|-C_0,
	\]
	where in the last inequality we used the fact that $C^*_{\tau,\eps}$ is a negative constant. 
	W.l.o.g., we can also assume that $|I|\geq5\eta_0$. Indeed, if the contrary holds, since $\tau\leq\tau_2$ one can use  \eqref{eq:Rgamma} and then \eqref{eq:stima1dper} holds with $C_0=-5\Upsilon$. 

Let us then set $I=[\bar x,\bar y]$ with $|\bar y-\bar x|\geq5\eta_0$. By the assumptions on $I$ and Corollary \ref{cor:pos}, whenever $\tau\leq\tau_2$ then there exist $s_0<t_0\in[\bar x-\delta_0/2, \bar x+\eta_0+\delta_0/2]$ with $|t_0-s_0|\leq\delta_0$ such that $|u_{x_i^\perp}(s_0)-u_{x_i^\perp}(t_0)|\geq(1-\delta)$ and $s_1<t_1\in[\bar y-\eta_0-\delta_0/2, \bar y+\delta_0/2]$ with $|t_1-s_1|\leq\delta_0$ such that $|u_{x_i^\perp}(s_1)-u_{x_i^\perp}(t_1)|\geq(1-\delta)$. In particular, there exist $\bar s\in[\bar x-\delta_0/2, \bar x+\eta_0+\delta_0/2]$ and $\bar t\in[\bar y-\eta_0-\delta_0/2, \bar y+\delta_0/2]$ such that $u_{x_i^\perp}(\bar s)=u_{x_i^\perp}(\bar t)=\frac12$.

Then we define a $|\bar t-\bar s|$-periodic function $\bar u^i$ on $\R$ by the usual symmetric reflection of $u_{|_{[\bar s,\bar t]}}$ (i.e. as in \eqref{eq:refl} with $h=|\bar t-\bar s|$) and set $\bar u(x)=\bar u^i(x_i)$ for all $x\in\R^d$. By optimality of the energy density $C^*_{\tau,\eps}$ on one-dimensional periodic functions one has that
\begin{equation*}
R_{i,\tau,\eps}(\bar u, x_i^\perp, [\bar s,\bar t))\geq C^*_{\tau,\eps}|\bar t-\bar s|\geq C^*_{\tau,\eps}|I|+\delta_0C^*_{\tau,\eps}.
\end{equation*}
Thus, since by the assumptions on $I$ one has that
\[
R_{i,\tau,\eps}(u, x_i^\perp, I)-R_{i,\tau,\eps}(u, x_i^\perp, [\bar s,\bar t])\geq-2\Upsilon, 
\]
we are left to prove that 
\begin{equation}\label{eq:perest}
R_{i,\tau,\eps}(u, x_i^\perp, [\bar s,\bar t])-R_{i,\tau,\eps}(\bar u, x_i^\perp, [\bar s,\bar t])\geq-\bar C.
\end{equation}

Indeed, one can further reduce \eqref{eq:perest} to prove that
\begin{equation}\label{eq:perest4}
R_{i,\tau,\eps}(u, x_i^\perp, [\bar x+3\eta_0/2,\bar y-3\eta_0/2])-R_{i,\tau,\eps}(\bar u, x_i^\perp, [\bar x+3\eta_0/2,\bar y-3\eta_0/2])\geq-\bar C,
\end{equation}
since by construction $[\bar x+3\eta_0/2,\bar y-3\eta_0/2]\subset[\bar s,\bar t]$ and $[\bar s,\bar t]\setminus[\bar x+3\eta_0/2,\bar y-3\eta_0/2]\subset [\bar x-\eta_0/2,\bar x+\eta_0+\delta_0]\cup[\bar y-\eta_0-\delta_0,\bar y+\eta_0/2]$, intervals on which (since $\delta_0\ll\eta_0/2$) by assumption $R_{i,\eps,\tau}\geq-\Upsilon$. Notice also that 
\begin{equation}\label{eq:infint}
\inf\{|s-t|:\,s\in[\bar x+3\eta_0/2,\bar y-3\eta_0/2],\,t\in \R\setminus [\bar s,\bar t]\}\geq\eta_0/2. 
\end{equation}

Since $u=\bar u$ on $[\bar s,\bar t]\supset [\bar x+3\eta_0/2,\bar y-3\eta_0/2]$, one has that
\begin{align}
R_{i,\tau,\eps}&(u, x_i^\perp, [\bar x+3\eta_0/2,\bar y-3\eta_0/2])-R_{i,\tau,\eps}(\bar u, x_i^\perp, [\bar x+3\eta_0/2,\bar y-3\eta_0/2])\notag\\
&\geq\int_{\bar x+3\eta_0/2}^{\bar y-3\eta_0/2}\int_{\R}\iint_{(s,\rho)\in \Omega(a,b)\cap\bigl(\{ s\notin[\bar s,\bar t]\}\cup\{s+\rho\notin [\bar s,\bar t]\}\bigr)}G^{-1}(\rho)\cdot\notag\\
&\cdot\Bigl(\Mi{s}{s+\rho}{i}-\overline{\mathcal M}^i_{\at}(\bar u, x_i^\perp,[s,s+\rho))\Bigr)\widehat K_\tau(\rho)\ds\d\rho\db\da\notag\\
&+\int_{\bar x+3\eta_0/2}^{\bar y-3\eta_0/2}\int_{\R}\iint_{(s,\rho)\in \Omega(a,b)\cap\bigl(\{ s\notin[\bar s,\bar t]\}\cup\{s+\rho\notin [\bar s,\bar t]\}\bigr)}G^{-1}(\rho)\cdot\Bigl((\bar u(s)-\bar u(s+\rho))^2-(u(s)-u(s+\rho))^2\Bigr)\cdot\notag\\
&\cdot\widehat K_\tau(\rho)\ds\d\rho\db\da.\label{eq:perest2}
\end{align}

Let us first estimate the second term in the r.h.s. of \eqref{eq:perest2}. 
Observe that 
	\begin{equation}\label{eq:intab}
\int_{\bar x+3\eta_0/2}^{\bar y-3\eta_0/2}\int_{\R}\chi_{\{(a,b):(s,\rho)\in\Omega(a,b)\}}(a,b)\da\db=|\rho||[\bar x+3\eta_0/2,\bar y-3\eta_0/2]\cap[s,s+\rho]|.
\end{equation}
 Moreover, since either $s$ or $s+\rho$ in the above integrals do not belong to $[\bar s,\bar t]$ and at the same time $[s,s+\rho]$ has to contain a point in $[\bar x+3\eta_0/2,\bar y-3\eta_0/2]$, then by \eqref{eq:infint} $|\rho|\geq\eta_0/2$.
 
 Then, using that 
 $|u|,|\bar u|\leq1$ and  Fubini Theorem one obtains 
\begin{align}
\int_{\bar x+3\eta_0/2}^{\bar y-3\eta_0/2}\int_{\R}&\iint_{(s,\rho)\in \Omega(a,b)}G^{-1}(\rho)\cdot\Bigl((\bar u(s)-\bar u(s+\rho))^2-(u(s)-u(s+\rho))^2\Bigr)\widehat K_\tau(\rho)\ds\d\rho\db\da\notag\\
&\geq-\iint_{\bigl(\{ s\notin[\bar s,\bar t]\}\cup\{s+\rho\notin [\bar s,\bar t]\}\bigr)\cap \{[s,s+\rho]\cap[\bar s,\bar t]\neq\emptyset\}}\frac{|[\bar x+3\eta_0/2,\bar y-3\eta_0/2]\cap[s,s+\rho]|}{\rho}\widehat K_\tau(\rho)\ds\d\rho\notag\\
&\geq-2\int_{\bar x+3\eta_0/2}^{\bar y-3\eta_0/2}\int_{-\infty}^{\bar x+\eta_0} \widehat K_\tau(s-t)\ds\dt\notag\\
&\geq-\bar C(\eta_0)
\end{align}
with $\bar C(\eta_0)\sim\frac{1}{\eta_0^{q-2}}$.

Let us now deal with the first term in the r.h.s. of \eqref{eq:perest2}. By positivity of the Modica-Mortola term, Fubini Theorem and \eqref{eq:intab}, one has that 
\begin{align}
&\int_{\bar x+3\eta_0/2}^{\bar y-3\eta_0/2}\int_{\R}\iint_{(s,\rho)\in \Omega(a,b)\cap\bigl(\{ s\notin[\bar s,\bar t]\}\cup\{s+\rho\notin [\bar s,\bar t]\}\bigr)}G^{-1}(\rho)\cdot\notag\\
&\cdot\Bigl(\Mi{s}{s+\rho}{i}-\overline{\mathcal M}^i_{\at}(\bar u, x_i^\perp,[s,s+\rho))\Bigr)\widehat K_\tau(\rho)\ds\d\rho\db\da\notag\\
&\geq
-\int_{\bar x+3\eta_0/2}^{\bar y-3\eta_0/2}\int_{\R}\iint_{(s,\rho)\in \Omega(a,b)\cap\bigl(\{ s\notin[\bar s,\bar t]\}\cup\{s+\rho\notin [\bar s,\bar t]\}\bigr)}G^{-1}(\rho)\overline{\mathcal M}^i_{\at}(\bar u, x_i^\perp,[s,s+\rho))\widehat K_\tau(\rho)\ds\d\rho\db\da\notag\\
&\geq-\iint_{\bigl(\{ s\notin[\bar s,\bar t]\}\cup\{s+\rho\notin [\bar s,\bar t]\}\bigr)\cap \{[s,s+\rho]\cap[\bar s,\bar t]\neq\emptyset\}}\overline{\mathcal M}^i_{\at}(\bar u, x_i^\perp,[s,s+\rho))\widehat K_\tau(\rho)\ds\d\rho.\label{eq:perest3}
\end{align}
Now notice that, by \eqref{eq:min0} and since $\bar u$ is obtained by periodic reflection of $u_{|_{[\bar s, \bar t]}}$, on each interval $[s,s+\rho]$ as above $R_{i,\eps,\tau}(\bar u,x_i^\perp,[s,s+\rho))<0$. Therefore, Corollary \ref{cor:pos0} holds implying that
\begin{equation}\label{eq:mest2}
\overline{\mathcal M}^i_{\at}(\bar u, x_i^\perp,[s,s+\rho))\leq\frac{2\Upsilon|\rho|}{\eta_0}.
\end{equation}
Substituting \eqref{eq:mest2} into \eqref{eq:perest3} one gets
\begin{align}
&-\iint_{\bigl(\{ s\notin[\bar s,\bar t]\}\cup\{s+\rho\notin [\bar s,\bar t]\}\bigr)\cap \{[s,s+\rho]\cap[\bar s,\bar t]\neq\emptyset\}}\overline{\mathcal M}^i_{\at}(\bar u, x_i^\perp,[s,s+\rho))\widehat K_\tau(\rho)\ds\d\rho\notag\\
&\geq-\int\int_{\bar x+3\eta_0/2}^{\bar y-3\eta_0/2}\int_{-\infty}^{\bar x+\eta_0} \frac{|s-t|}{\eta_0}\widehat K_\tau(s-t)\ds\dt\notag\\
&\geq-\hat C(\eta_0),
\end{align}
with $\hat C(\eta_0)\sim\frac{1}{\eta_0^{q-3}}$, thus concluding the proof of \eqref{eq:perest4}, hence of the lemma.

\end{proof}

In the next lemma, diffuse interface counterpart of Lemma 7.11 in \cite{DR}, we give a lower bound of the energy on cubes where $u$ is either close to $0$ or close to $1$ in $L^1$. 

\begin{lemma}\label{lemma:stimaa-1}
	There exists a constant $C_1>0$ such that the following holds.  Let $u\in W^{1,2}_{\loc}(\R^d;[0,1])$ be such that
	\begin{equation}
	\min\bigl\{\|u-1\|_{L^1(Q_l(z))},\,\|u\|_{L^1(Q_l(z))} \bigr\}\leq\nu l^d, 
	\end{equation}
	for some $\nu>0$. Let $1<\Upsilon\leq{\frac{17}{16}}$. Then, provided $\tau$ is sufficiently small,
		\begin{equation}\label{eq:fa-1}
	\begin{split}
	\bar F_{\tau,\eps} (u,Q_{l}(z)) \geq -C_1\frac {\Upsilon\nu d } {\eta_0 },
	\end{split}
	\end{equation}
	where $\eta_0=\eta_0(\Upsilon)$ is as in Corollary \ref{cor:pos}.
\end{lemma}

\begin{proof}
	By assumption, we assume w.l.o.g. that $\|u-1\|_{L^1(Q_l(z))}\leq\nu l^d$. In particular, by Chebyshev inequality 
	\begin{equation}
	\label{eq:38}
	\Bigl|\Bigl\{u\leq\frac38\Bigr\}\cap Q_l(z)\Bigr|\leq \frac{8}{5}\nu l^d.
	\end{equation}
	Let $z_i-l/2=t_0<t_1<\dots<t_N=z_i+l/2$ be a partition of $Q_l^i(z_i)$ into intervals $[t_k,t_{k+1})$ of size $\eta_0$ (with eventually $|t_N-t_{N-1}|\leq\eta_0$). Then, 
	\begin{align}
	\bar{F}_{i,\tau,\eps}(u,Q_{l}(z))        & \geq \frac{1}{l^d} \int_{Q^{\perp}_l(z^{\perp}_{i})}R_{i,\tau,\eps}(u,x_i^\perp,Q_l^i(z_i))\dx_i^\perp\notag\\
	&\geq \frac{1}{l^d} \int_{Q^{\perp}_l(z^{\perp}_{i})}\sum_{k\in\mathcal K(Q_l^i(z))}R_{i,\tau,\eps}(u,x_i^\perp,[t_k,t_{k+1}))\dx_i^\perp,\label{eq:a-11}
	\end{align}
	\[
	\mathcal K(Q_l^i(z))=\{k\in\{1,\dots,N\}:\,R_{i,\tau,\eps}(u,x_i^\perp,[t_k,t_{k+1}))<0\}.
	\]
	 In particular, by Corollary \ref{cor:pos}, given $1<\Upsilon\leq\frac{17}{16}$ and ${\delta}\ll1$, there exist $\eta_0$ and $\tau_2$ such that for any $\tau\leq\tau_2$
	\begin{equation}\label{eq:ups}R_{i,\tau,\eps}(u,x_i^\perp,([t_k,t_{k+1})))\geq-\Upsilon
	\end{equation} and  there exist $\bar s<\bar t\in[t_k-\delta_0/2, t_{k+1}+\delta_0/2]$ with $|\bar t-\bar s|\leq\delta_0\ll\eta_0$ such that $|u_{x_i^\perp}(\bar s)-u_{x_i^\perp}(\bar t)|\geq(1- \delta)$. Moreover, 
	\begin{align}
	&\forall\,t\in[\bar t, t_{k+1}+\eta_0/2),\quad|u_{x_i^\perp}(t)-u_{x_i^\perp}(\bar t)|\leq\frac{1}{4}+\sqrt{2\delta}\label{eq:close3}\\
	&\forall\,s\in[t_k-\eta_0/2,\bar s],\quad|u_{x_i^\perp}(s)-u_{x_i^\perp}(\bar s)|\leq\frac{1}{4}+\sqrt{2\delta}.
	\end{align}
	Hence, when ${\delta}$ is sufficiently small, for any $k\in\mathcal K(Q_l^i(z_i))$ there exist an interval of size at least $\eta_0/4$ on which $u\leq\frac{3}{8}$. By construction, there exist at least $\frac12\#\mathcal K(Q_l^i(z))$ of such intervals which are disjoint. Then, inserting this information together with \eqref{eq:ups} and \eqref{eq:38} in \eqref{eq:a-11} one obtains that
	\begin{align}
		\bar{F}_{i,\tau,\eps}(u,Q_{l}(z))&\geq\frac{1}{l^d} \int_{Q^{\perp}_l(z^{\perp}_{i})}\sum_{k\in\mathcal K(Q_l^i(z))}R_{i,\tau,\eps}(u,x_i^\perp,[t_k,t_{k+1}))\dx_i^\perp\notag\\
		&\geq-\frac{1}{l^d}\int_{Q^{\perp}_l(z^{\perp}_{i})}{\Upsilon}\#\mathcal K(Q_l^i(z_i))\dx_i^\perp\notag\\
		&\geq-\frac{1}{l^d}\int_{Q^{\perp}_l(z^{\perp}_{i})}{\Upsilon}\frac{2\Bigl|\Bigl\{u_{x_i^\perp}\leq\frac38\Bigr\}\cap Q_l(z)\Bigr|}{\eta_0/4}\notag\\
		&\geq-C_1\frac{\Upsilon\nu}{\eta_0}.
	\end{align}
Summing over $i\in\{1,\dots,d\}$ one obtains \eqref{eq:fa-1}.	
\end{proof}

The following lemma roughly shows that, whenever the function $u$ on a subset of a slice in direction $e_i$ is close to a stripe with boundaries orthogonal to $e_j$ for some $j\neq i$, then the contribution to the energy $\bar F_{i,\tau,\eps}$ is positive. It is the counterpart of the local stability Lemma 7.8 in \cite{DR}. 

\begin{lemma}
	\label{lemma:stability}
	Let $\eta_0$, $\tau_2$ be as in Corollary \ref{cor:pos}. Then, there exists $0<\tau_3\leq\tau_2$, $\bar\sigma>0$ (independent of $l$) such that for every $\tau\leq\tau_3$ and  $\sigma\leq\bar{\sigma}$ the following holds: let $x_i^\perp\in Q_l^\perp(z)$ and $z$ s.t. 
	\begin{equation}\label{eq:djsmall}
	D^{j}_{\eta}(u,Q_l(z))\leq\sigma \text{ for some } j\neq i.
	\end{equation}
	Then,
	\begin{align}\label{eq:r+v}
	R_{i,\tau,\eps}(u,x_i^\perp,(z_i-l/2+\eta_0,z_i+l/2-\eta_0))+V_{i,\tau}(u,x_i^\perp,(z_i-l/2+\eta_0,z_i+l/2-\eta_0))\geq0
	\end{align}
{and equality holds if and only if $u=u(x_i^\perp)$.}
\end{lemma}

\begin{proof}
	Let $z_i-l/2+\eta_0=t_0<t_1<\dots<t_N=z_i+l/2-\eta_0$ be a partition of $(z_i-l/2+\eta_0,z_i+l/2-\eta_0)$ into intervals of lenght $\eta_0$ (with possibly $|t_N-t_{N-1}|\leq\eta_0$). Let 
	\[
	\mathcal K(Q^i_{l+2\eta_0}(z))=\{k\in\{1,\dots,N\}:\,R_{i,\eps,\tau}(u,x_i^\perp,[t_k,t_{k+1}))<0\}.
	\]
	By Corollary \ref{cor:pos}, for every $k\in\mathcal K(Q^i_{l+2\eta_0}(z))$ one has the following:
	\begin{itemize}
		\item $\Mi{t_k-\eta_0/2}{t_{k+1}+\eta_0/2}{i}\leq\Upsilon$\\
		\item $R_{i,\eps,\tau}(u,x_i^\perp,[t_k,t_{k+1}))\geq-\Upsilon$\\
		\item $\exists\,\bar s<\bar t\in[t_k-\delta_0,t_{k+1}+\delta_0]$ with $|\bar s-\bar t|\leq\delta_0$ and $|u_{x_i^\perp}(\bar s)-u_{x_i^\perp}(\bar t)|>1-\delta$.
			\end{itemize}
		 In particular, assuming w.l.o.g. that $u_{x_i^\perp}(\bar s)>u_{x_i^\perp}(\bar t)$, since $\delta_0\ll\eta_0/8$ and $\Mi{t_k-\eta_0/2}{t_{k+1}+\eta_0/2}{i}\leq\Upsilon$, by \eqref{eq:381} and \eqref{eq:581}  one has that, taking $\delta$ sufficiently small,
		\begin{align}
		u_{x_i^\perp}&\geq\frac58\quad\text{on $[t_k-\eta_0/2,\bar s]$},\label{eq:582}\\
		u_{x_i^\perp}&\leq\frac38\quad\text{on $[\bar t,t_{k+1}+\eta/2]$}.\label{eq:382}
		\end{align}

		Hence, using the positivity of $f_u$, Fubini Theorem and the fact that
			\begin{equation}\label{eq:tkk}
		\int_{t_k-\eta_0/2}^{t_{k+1}+\eta_0/2}\int_{\R}\chi_{\{(a,b):(s,\rho)\in\Omega(a,b)\}}(a,b)\da\db=|\rho||[t_k-\eta_0/2,t_{k+1}+\eta_0/2]\cap[s,s+\rho]|,
		\end{equation}
		one has that
		\begin{align*}
		V_{i,\tau}(u,x_i^\perp,&[t_k-\eta_0/2,t_{k+1}+\eta_0/2))\geq\frac{1}{2d}\int_{t_k-\eta_0/2}^{t_{k+1}+\eta_0/2}\int_{\R}\iint_{(s,\rho)\in\Omega(a,b)\cap\Omega(\bar s,\bar t)}\int_{\R^{d-1}}G^{-1}(\rho)\cdot\notag\\
		&\cdot f_u(x_i^\perp,s,\zeta_i^\perp,s+\rho)K_\tau(\rho e_i+\zeta_i^\perp)\d\zeta_i^\perp\ds\d\rho\db\da\notag\\
		&\geq \frac{1}{2d}\iint_{(s,\rho)\Omega(\bar s,\bar t), \{\frac{\eta_0}{8}\leq|\rho|\leq\frac{\eta_0}{4}\}}\frac{|[t_k-\eta_0/2,t_{k+1}+\eta_0/2]\cap[s,s+\rho]|}{|\rho|}\cdot\notag\\
		&\cdot\int_{\R^{d-1}}f_u(x_i^\perp,s,\zeta_i^\perp,s+\rho)K_\tau(\rho e_i+\zeta_i^\perp)\d\zeta_i^\perp\ds\d\rho\notag\\
		&\geq\frac{1}{16d}\iint_{(s,\rho)\Omega(\bar s,\bar t), \{\frac{\eta_0}{8}\leq|\rho|\leq\frac{\eta_0}{4}\}}\int_{\R^{d-1}}f_u(x_i^\perp,s,\zeta_i^\perp,s+\rho)K_\tau(\rho e_i+\zeta_i^\perp)\d\zeta_i^\perp\ds\d\rho.
		\end{align*}
		
		Now observe that, for $(s,\rho)\in\Omega(\bar s,\bar t)$ with $\frac{\eta_0}{8}\leq|\rho|\leq\frac{\eta_0}{4}$, due to \eqref{eq:582} and \eqref{eq:382} one has that
		\begin{equation*}
		f_u(x_i^\perp,s,\zeta_i^\perp,s+\rho)\geq\Big(\frac{1}{4}- \big[ u(x_i^\perp + \zeta_i^\perp + se_i ) - u(x_i^\perp + \zeta_i^\perp + (s+\rho)e_i)  \big]\Big)^2.
		\end{equation*}
		Hence, 
		\begin{align*}
		\frac{1}{16d}&\iint_{(s,\rho)\Omega(\bar s,\bar t), \{\frac{\eta_0}{8}\leq|\rho|\leq\frac{\eta_0}{4}\}}\int_{\R^{d-1}}f_u(x_i^\perp,s,\zeta_i^\perp,s+\rho)K_\tau(\rho e_i+\zeta_i^\perp)\d\zeta_i^\perp\ds\d\rho\notag\\
		&\geq\frac{1}{16d}\int_{\bar s-\alpha}^{\bar s}\int_{\bar t-s}^{\bar t-s+\alpha}\int_{|\zeta_i^\perp|<\alpha}\Big(\frac{1}{4}- \big[ u(x_i^\perp + \zeta_i^\perp + se_i ) - u(x_i^\perp + \zeta_i^\perp + (s+\rho)e_i)  \big]\Big)^2K_\tau(\rho e_i+\zeta_i^\perp)\d\zeta_i^\perp\d\rho\ds\notag\\
		&\geq\frac{1}{16d}\int_{\bar s-\alpha}^{\bar s}\int_{\bar t-s}^{\bar t-s+\alpha}\int_{|\zeta_i^\perp|<\alpha}\frac{\Big(\frac{1}{4}- \big[ u(x_i^\perp + \zeta_i^\perp + se_i ) - u(x_i^\perp + \zeta_i^\perp + (s+\rho)e_i)  \big]\Big)^2}{(3\alpha+\delta_0+\tau^{1/\beta})^p}\d\zeta_i^\perp\d\rho\ds
		\end{align*}
	where $\alpha<\eta_0/4$, and in the last inequality we used that $|\rho|\leq2\alpha+\delta_0$. 
	
	By Lemma \ref{lemma:int14}, if $\delta_0\leq\alpha$ (as in this case, since $\delta_0\ll\eta_0/8$) and $\sigma$ is sufficiently small, 
	\begin{align*}
	\frac{1}{16d}\int_{\bar s-\alpha}^{\bar s}\int_{\bar t-s}^{\bar t-s+\alpha}&\int_{|\zeta_i^\perp|<\alpha}\frac{\Big(\frac{1}{4}- \big[ u(x_i^\perp + \zeta_i^\perp + se_i ) - u(x_i^\perp + \zeta_i^\perp + (s+\rho)e_i)  \big]\Big)^2}{(3\alpha+\delta_0+\tau^{1/\beta})^p}\d\zeta_i^\perp\d\rho\ds>\notag\\
	&>\frac{\alpha^{d+1}}{8\cdot16d(3\alpha+\delta_0+\tau^{1/\beta})^p}.
	\end{align*} 
	Then, assuming that $0<\tau\leq\tau_3$ is such that $\alpha\geq\tau_3^{1/\beta}$ one has that $1/(3\alpha+\delta_0+\tau^{1/\beta})^p\geq1/(5\alpha)^p$ and thus since $p\geq d+2$ one has that 
	\begin{align*}
		V_{i,\tau}(u,x_i^\perp,&[t_k-\eta_0/2,t_{k+1}+\eta_0/2))\geq\frac{1}{8\cdot16d\cdot 5^p\alpha}.
	\end{align*}	
	To conclude, we observe that
	\begin{align*}
	R_{i,\tau,\eps}(u,x_i^\perp,&[t_k-\eta_0/2,t_{k+1}+\eta_0/2))+	V_{i,\tau}(u,x_i^\perp,[t_k-\eta_0/2,t_{k+1}+\eta_0/2))\geq-\Upsilon+\frac{1}{8\cdot 16d\cdot5^p\alpha}
	\end{align*}
	and the r.h.s. of the above inequality is strictly positive provided $\alpha$ is chosen such that
	\[
	-\Upsilon+\frac{1}{8\cdot16d\cdot 5^p\alpha}>0.
	\]
	Thus, \eqref{eq:r+v} holds. {Moreover, since whenever $R_{i,\tau,\eps}<0$ then by the above $R_{i,\tau,\eps}+V_{i,\tau}>0$, one has that equality in \eqref{eq:r+v} holds if and only if $R_{i,\tau,\eps}=0$ and $V_{i,\tau}=0$, which implies $u=u(x_i^\perp)$. }

\end{proof}

In the following lemma we estimate from below the contributions on intervals  of slices in directions $e_i$ to the energy $\bar F_{i,\tau,\eps}$. In order to obtain such estimates, we will use the various lemmas and corollaries proved in this section. 

  \begin{lemma}\label{lemma:stimalinea}
	Let $\bar\sigma$, ${\tau_3}>0$ as in Lemma \ref{lemma:stability}. Let $\sigma\leq\bar{\sigma}$, $\tau\leq{\tau_3}$ and $l>C_0/(-C^*_{\tau,\eps})$, where $C_0$ is the constant appearing in Lemma \ref{lemma:1dopt}.  Let $z\in[0,L)^d$ and let $\bar\eta>0$.

	The following statements hold: there exists $M_0$ constant independent of $l$ (but depending on the dimension) such that  
	\begin{enumerate}[(i)]
		\item Let $J\subset \R$ an interval  such that for every $s\in J$ one has that  $D^{j}_{\bar\eta}(u,Q_{l}(z_i^\perp+se_i))\leq \sigma$ with $j\neq i$. 
		Then
		\begin{equation}
		\label{eq:gstr20}
		\begin{split}
		\int_{J} \bar{F}_{i,\tau,\eps}(u,Q_{l}(z^{\perp}_{i}+se_i))\ds \geq - \frac{M_{0}}{l}.
		\end{split}
		\end{equation}
		Moreover, if $J = [0,L)$, then 
		\begin{equation}
		\label{eq:gstr21}
		\begin{split}
		\int_{J} \bar{F}_{i,\tau,\eps}(u,Q_{l}(z^{\perp}_{i}+se_i))\ds \geq0
		\end{split}
		\end{equation}
{	and equality holds only if $u=u(x_i^\perp)$. }
		\item Let $J = (a,b)\subset \R$. 
		If for $s=a$ and $s=b$ it holds $D_{\bar \eta}^j(u,Q_{l}(z^{\perp}_i+se_i)) \leq \sigma$ with $j\neq i$, then 
		\begin{equation}
		\label{eq:gstr27}
		\begin{split}
		\int_{J} \bar{F}_{i,\tau,\eps}(u,Q_{l}(z^{\perp}_{i}+se_i))\ds \geq | J| C^{*}_{\tau,\eps} -\frac{M_0} l,
		\end{split}
		\end{equation}
		otherwise
		\begin{equation}
		\label{eq:gstr36}
		\begin{split}
		\int_{J} \bar{F}_{i,\tau,\eps}(u,Q_{l}(z^{\perp}_{i}+se_i))\ds \geq | J| C^{*}_{\tau,\eps} - M_0l.
		\end{split}
		\end{equation}
		Moreover, if $J = [0,L)$, then
		\begin{equation}
		\label{eq:gstr28}
		\begin{split}
		\int_{J} \bar{F}_{i,\tau,\eps}(u,Q_{l}(z^{\perp}_{i}+se_i))\ds \geq | J| C^{*}_{\tau,\eps}.
		\end{split}
		\end{equation}
	\end{enumerate}
\end{lemma}

\begin{proof} Given Lemmas \ref{lemma:rpos1}-\ref{lemma:stability}, the proof proceeds in a way similar to that followed in Lemma 7.9 of \cite{DR}.

	Let us first prove (i). For simplicity of notation, we may assume  without loss of generality that $a=0$ and $b=l'$, namely $J= [0,l')$.   For any $x_i,y_i\in[0,L)$ we also set
	\begin{align*}
	R_{i,\tau,\eps}(u,x_i^\perp,(x_i,y_i))&=\int_{x_i}^{y_i}r_{i,\tau,\eps}(u,x_i^\perp,a)\da,\notag\\
		V_{i,\tau}(u,x_i^\perp,(x_i,y_i))&=\int_{x_i}^{y_i}v_{i,\tau}(u,x_i^\perp,a)\da.
	\end{align*}
	 
	From the definition of $\bar{F}_{i,\tau,\eps}$ given in \eqref{eq:fbari} and since $W_{i,\tau}\geq0 $, we have that
	
	\begin{equation}
	\label{eq:gstr25}
	\begin{split}
	\int_{J} \bar{F}_{i,\tau,\eps}(u,Q_{l}(z^{\perp}_i+se_i)) \ds &\geq  \frac{1}{l^{d}} 
	\int_{J} \int_{Q^{\perp}_{l}(z^{\perp}_{i})} \Bigl\{R_{i,\tau,\eps}(u,x_i^\perp,Q_l^i(s))+V_{i,\tau}(u,x_i^\perp,Q_l^i(s))\Bigr\} \dx^{\perp}_{i} \ds \\ 
	&=  \frac{1}{l^{d}} 
	\int_{J} \int_{Q^{\perp}_{l}(z^{\perp}_{i})}\int_{s-l/2}^{s+l/2}\bigl\{r_{i,\tau,\eps}(u,x_i^\perp,a)+v_{i,\tau}(u,x_i^\perp,a)\bigr\}\da \dx^{\perp}_{i} \ds \\
	&= \frac{1}{l^{d-1}} \int_{Q^{\perp}_{l}(z^{\perp}_{i})} \int_{-l/2}^{l'+l/2}\frac{\bigl|[a-l/2,a+l/2]\cap[0,l']\bigr|}{l}\bigl\{r_{i,\tau,\eps}(u,x_i^\perp,a)+v_{i,\tau}(u,x_i^\perp,a)\bigr\}\da,
	\end{split}
	\end{equation}
	where in order to obtain the last line we used Fubini Theorem. 
	
	Let us  now estimate the last term in \eqref{eq:gstr25}.
	
	One has that
	\begin{align*}
	&\frac{1}{l^{d-1}} \int_{Q^{\perp}_{l}(z^{\perp}_{i})} \int_{-l/2}^{l'+l/2}\frac{\bigl|[a-l/2,a+l/2]\cap[0,l']\bigr|}{l}\bigl\{r_{i,\tau,\eps}(u,x_i^\perp,a)+v_{i,\tau}(u,x_i^\perp,a)\bigr\}\da\notag\\
	&=\frac{1}{l^{d-1}} \int_{Q^{\perp}_{l}(z^{\perp}_{i})} \int_{[-l/2,-l/2+\eta_0]\cup[l'+l/2-\eta_0,l'+l/2]}\frac{\bigl|[a-l/2,a+l/2]\cap[0,l']\bigr|}{l}\bigl\{r_{i,\tau,\eps}(u,x_i^\perp,a)+v_{i,\tau}(u,x_i^\perp,a)\bigr\}\da\notag\\
	&	+\frac{1}{l^{d-1}} \int_{Q^{\perp}_{l}(z^{\perp}_{i})} \int_{-l/2+\eta_0}^{l'+l/2-\eta_0}\frac{\bigl|[a-l/2,a+l/2]\cap[0,l']\bigr|}{l}\bigl\{r_{i,\tau,\eps}(u,x_i^\perp,a)+v_{i,\tau}(u,x_i^\perp,a)\bigr\}\da\notag\\
	&\geq-2\Upsilon\frac{\eta_0}{l}+\frac{1}{l^{d-1}} \int_{Q^{\perp}_{l}(z^{\perp}_{i})} \int_{-l/2+\eta_0}^{l'+l/2-\eta_0}\frac{\bigl|[a-l/2,a+l/2]\cap[0,l']\bigr|}{l}\bigl\{r_{i,\tau,\eps}(u,x_i^\perp,a)+v_{i,\tau}(u,x_i^\perp,a)\bigr\}\da,
	\end{align*}
where in the last inequality we used \eqref{eq:Rgamma}.
	As in Lemma \ref{lemma:stability}, one can now partition $[-l/2+\eta_0,l'+l/2-\eta_0]$ in intervals of size $\eta_0$ and notice that on the intervals $[t_k,t_{k+1})$ of this partition such that $R_{i,\tau,\eps}(u,x_i^\perp,[t_k,t_{k+1}))=\int_{t_k}^{t_{k+1}}r_{i,\tau,\eps}(u,x_i^\perp,a)\da<0$, then 
	\begin{equation*}
	\int_{t_k}^{t_{k+1}}\frac{\bigl|[a-l/2,a+l/2]\cap[0,l']\bigr|}{l}r_{i,\tau,\eps}(u,x_i^\perp,a)\da\geq-\Mi{t_k}{t_{k+1}}{i}\geq-\Upsilon. 
	\end{equation*}
	On the other hand, 
	\begin{align*}
	\int_{t_k}^{t_{k+1}}\frac{\bigl|[a-l/2,a+l/2]\cap[0,l']\bigr|}{l}v_{i,\tau}(u,x_i^\perp,a)\da&\geq \frac{\eta_0}{l}\int_{t_k}^{t_{k+1}}v_{i,\tau}(u,x_i^\perp,a)\da\notag\\
	&\geq\frac{\eta_0}{l} V_{i,\tau}(u,x_i^\perp,[t_k,t_{k+1})),
	\end{align*}
	and provided $\bar{\sigma}$ and $\tau_3$ are sufficiently small, since for every $s\in [0,l']$ one has that  $D^{j}_{\bar\eta}(u,Q_{l}(z_i^\perp+se_i))\leq \sigma$ with $j\neq i$, as in Lemma \ref{lemma:stability} one has that
	\[
	-\Upsilon+\frac{\eta_0}{l} V_{i,\tau}(u,x_i^\perp,[t_k,t_{k+1}))>0,
	\]	
	thus proving \eqref{eq:gstr20} with $M_0=\frac{2\Upsilon\eta_0}{l}$.

	Let us now prove  \eqref{eq:gstr21}. In this case, by periodicity of $u$ w.r.t. $[0,L)^d$ and Fubini Theorem,
	\begin{equation*}
	\begin{split}
	\int_{0}^L \bar{F}_{i,\tau,\eps}(u,Q_{l}(z^{\perp}_i+se_i)) \ds &\geq  \frac{1}{l^{d}} 
	\int_{0}^L \int_{Q^{\perp}_{l}(z^{\perp}_{i})}\Bigl\{R_{i,\eps,\tau}(u,x_i^\perp, Q_l^i(s))+V_{i,\tau}(u,x_i^\perp, Q_l^i(s))\Bigr\}\dx^{\perp}_{i} \ds \\ 
 &= \frac{1}{l^{d-1}} \int_{Q^{\perp}_{l}(z^{\perp}_{i})}\Bigl\{R_{i,\eps,\tau}(u,x_i^\perp, [0,L))+V_{i,\tau}(u,x_i^\perp, [0,L))\Bigr\}\dx^{\perp}_{i}.
	\end{split}
	\end{equation*}

	Hence \eqref{eq:gstr21} holds since we can apply directly Lemma \ref{lemma:stability} without having points close to the boundary.

	Let us now prove (ii).  W.l.o.g. let us assume that $J= (0,l')$.  
	
	As in \eqref{eq:gstr25} one has that 
	
	\begin{align*}
	\int_{J}& \bar{F}_{i,\tau,\eps}(u,Q_{l}(z^{\perp}_i+se_i)) \ds \geq  \frac{1}{l^{d}} 
	\int_{J} \int_{Q^{\perp}_{l}(z^{\perp}_{i})} \Bigl\{R_{i,\tau,\eps}(u,x_i^\perp,Q_l^i(s))+V_{i,\tau}(u,x_i^\perp,Q_l^i(s))\Bigr\} \dx^{\perp}_{i} \ds \notag\\ 
	&= \frac{1}{l^{d-1}} \int_{Q^{\perp}_{l}(z^{\perp}_{i})} \int_{-l/2}^{l'+l/2}\frac{\bigl|[a-l/2,a+l/2]\cap[0,l']\bigr|}{l}\bigl\{r_{i,\tau,\eps}(u,x_i^\perp,a)+v_{i,\tau}(u,x_i^\perp,a)\bigr\}\da\notag\\
	&=\frac{1}{l^{d-1}} \int_{Q^{\perp}_{l}(z^{\perp}_{i})} \int_{[-l/2,l/2]\cup[l'-l/2,l'+l/2]}\frac{\bigl|[a-l/2,a+l/2]\cap[0,l']\bigr|}{l}\bigl\{r_{i,\tau,\eps}(u,x_i^\perp,a)+v_{i,\tau}(u,x_i^\perp,a)\bigr\}\da\notag\\
	&	+\frac{1}{l^{d-1}} \int_{Q^{\perp}_{l}(z^{\perp}_{i})} \int_{l/2}^{l'-l/2}\bigl\{r_{i,\tau,\eps}(u,x_i^\perp,a)+v_{i,\tau}(u,x_i^\perp,a)\bigr\}\da.
	\end{align*}

As in the proof of \eqref{eq:gstr20}, if for $s=0$ and $s=l'$ it holds $D^j_{\bar{\eta}}(u,Q_l(z_i^\perp+se_i))\leq\sigma$, then  one has that
\begin{align}
\frac{1}{l^{d-1}} &\int_{Q^{\perp}_{l}(z^{\perp}_{i})} \int_{[-l/2,l/2]\cup[l'-l/2,l'+l/2]}\frac{\bigl|[a-l/2,a+l/2]\cap[0,l']\bigr|}{l}\bigl\{r_{i,\tau,\eps}(u,x_i^\perp,a)+v_{i,\tau}(u,x_i^\perp,a)\bigr\}\da\notag\\
&\geq-4\frac{\Upsilon\eta_0}{l}.\label{eq:361}
\end{align}

On the other hand, by Lemma \ref{lemma:1dopt} and the assumption $l>C_0/(-C^*_{\tau,\eps})$

\begin{align}
\frac{1}{l^{d-1}} &\int_{Q^{\perp}_{l}(z^{\perp}_{i})} \int_{l/2}^{l'-l/2}\bigl\{r_{i,\tau,\eps}(u,x_i^\perp,a)+v_{i,\tau}(u,x_i^\perp,a)\bigr\}\da\notag\\
&\geq\frac{1}{l^{d-1}} \int_{Q^{\perp}_{l}(z^{\perp}_{i})}R_{i,\tau,\eps}(u,x_i^\perp,(l/2,l'-l/2))\dx_i^\perp\notag\\
&\geq C^*_{\tau,\eps}|J|-lC^*_{\tau,\eps}-C_0\notag\\
&\geq C^*_{\tau,\eps}|J|.\label{eq:362}
\end{align}

Thus, \eqref{eq:gstr27} follows combining \eqref{eq:361} and \eqref{eq:362}. 

If instead either for $s=0$ or for  $s=l'$ it holds $D^j_{\bar{\eta}}(u,Q_l(z_i^\perp+se_i))>\sigma$, then we partition the intervals $[-l/2,l/2]\cup[l'-l/2,l'+l/2]$ into intervals of size $\eta_0$, on which by \eqref{eq:Rgamma} one has that  $R_{i,\tau,\eps}\geq-\Upsilon$. In this way we get

\begin{align}
\frac{1}{l^{d-1}} &\int_{Q^{\perp}_{l}(z^{\perp}_{i})} \int_{[-l/2,l/2]\cup[l'-l/2,l'+l/2]}\frac{\bigl|[a-l/2,a+l/2]\cap[0,l']\bigr|}{l}\bigl\{r_{i,\tau,\eps}(u,x_i^\perp,a)+v_{i,\tau}(u,x_i^\perp,a)\bigr\}\da\notag\\
&\geq-2\Upsilon \frac{l}{\eta_0},\label{eq:363}
\end{align}
being $\frac{l}{\eta_0}$ an upper bound for the number of disjoint  intervals of length $\eta_0$ inside an interval of length $l$.  

Thus, \eqref{eq:gstr36} follows from \eqref{eq:362} and \eqref{eq:363}.

	The proof of \eqref{eq:gstr28} proceeds using the $L$-periodicity of the contributions as done for \eqref{eq:gstr21}. 
	
\end{proof}

\section{Proof of Theorem \ref{thm:main}}\label{sec:mainthm} 

In this section we complete the proof of Theorem \ref{thm:main}, bringing together the lemmas of the previous section in order to show the optimality of one-dimensional periodic functions in a range of $\tau,\eps$ independent of $L$. This part of the proof follows closely the strategy adopted in Section 7 of \cite{DR}.

The sets defined in the proof and the main estimates will depend on a set of parameters $l,\Upsilon, \eta_0,\sigma,\rho,M,\hat{\eta},\tau$ and $\eps$. If suitably chosen, they lead to the proof Theorem \ref{thm:main}. 

Let us first specify how the parameters are chosen, and their dependence on each other. The reason for such choices will be clarified during the proof.

Let $0<-C^*<-C^*_{\tau,\eps}$ for all sufficiently small $\eps$ and $\tau$. Such a $C^*$ exists by the $\Gamma$-convergence of the functionals $\Fcal_{\tau,L,\eps}$ to a functional which is finite only on stripes as $\tau,\eps\to0$.  For the sharp interface problem (i.e., the $\Gamma$-limit as $\eps\to0$) one can also explicitly compute such constants and see this directly (see \cite{Ker}).

We fix a family of parameters as follows:
\begin{itemize}
		\item Let $1<\Upsilon\leq\frac{17}{16}$ and let $\eta_0=\eta_0(\Upsilon)$ as in Lemma \ref{lemma:rpos1}.
	
	\item  Then we  fix $l>0$ s.t.
	\begin{equation}\label{eq:lfix}
	l>\max \Big \{ \frac{dC(d,\Upsilon,\eta_0)}{-C^*}, \frac{C_0}{-C^*}\Big\}, 
	\end{equation}
	where $C(d,\Upsilon,\eta_0)$ is a constant (depending on $d$, $\Upsilon$ and $\eta_0$) that appears in \eqref{eq:toBeShown_integral}, and
	$C_0$ is the constant which appears in the statement of Lemma \ref{lemma:1dopt}.

	\item Choose $\nu=\frac{1}{l}$. 
	
	\item Let $\sigma\leq\min\bigl\{\sigma_0,\bar \sigma\bigr\}$ with $\sigma_0=\sigma_0(\nu)$ as in \eqref{eq:gsmstr2} and $\bar \sigma$ as in Lemma \ref{lemma:stability}.
	
		\item Thanks to Remark \ref{rmk:lip} (i), we then fix
	\begin{equation}\label{eq:rhofix}
		\rho\sim\sigma l. 
	\end{equation}
	in such a way that  we have that  for any $\eta$ the following holds
	\begin{equation}\label{eq:rhofix2}
		\forall\,z,z'\text{ s.t. }D_{\eta}(E,Q_l(z))\geq\sigma,\:|z-z'|_\infty\leq\rho\quad\Rightarrow\quad D_\eta(E,Q_l(z'))\geq\sigma/2.
	\end{equation}

		\item Then we fix $M$ such that
	\begin{equation}
		\label{eq:Mfix}
		\frac{M\rho}{2d}>M_0l, 
	\end{equation}
where $M_0$ is the constant appearing in Lemma \ref{lemma:stimalinea}.

	\item Finally, let $\tau>0$, $\eps>0$ satisfy the following
	\begin{align}
		&\tau\leq\tau_3\quad \text{ as in Lemma \ref{lemma:stability}}\\
		&\text{$\tau$ is such that Lemma \ref{lemma:stimaa-1} holds}\\
		&\tau\leq\hat{\tau},\,\eps\leq\hat{\eps}\text{ as in Proposition \ref{lemma:local_rigidity}}.
		\end{align}
	\item We fix also $\hat\eta=\hat \eta(M,l)$ as in Proposition \ref{lemma:local_rigidity}.

\end{itemize}

Given such parameters, let us prove Theorem \ref{thm:main} for any $L>l$ of the form $L=2kh^*_{\tau,\eps}$, with $k\in\N$.

Let $u$ be a minimizer of $\Fcal_{\tau,L,\eps}$. Since $u$ is $[0,L)^d$-periodic, we can consider $u$ as defined on $\T^d_L$, where $\T^d_L$ is the $d$-dimensional torus of size $L$. Thus the problem is naturally defined on the torus. Hence
with a slight abuse of notation, we will denote by $[0,L)^d$  the cube of size $L$ with the usual identification of the boundary.

{\bf Decomposition of $[0,L)^d$: }

We define
\begin{equation*}
\begin{split}
\tilde{A}_{0}:= \insieme{ z\in [0,L)^d:\ D_{\eta}(u,Q_{l}(z)) \geq \sigma }.
\end{split}
\end{equation*}
Hence, by Lemma \ref{lemma:local_rigidity}, for every $z\in \tilde{A}_{0}$ one has that $\bar{F}_{\tau,\eps}(u,Q_{l}(z)) > M$. 

Let us denote by $\tilde{A}_{-1}$ the set of points
\begin{equation*}
\begin{split}
\tilde{A}_{-1}: = \insieme{z\in [0,L)^d: \exists\, i,j \text{ with } i\neq j \text{ \st }\, D^{i}_{\eta} (u,Q_{l}(z))\leq\sigma , D^{j}_{\eta} (u,Q_{l}(z)) \leq \sigma }.
\end{split}
\end{equation*}
One can easily see that $\tilde A_0$ and $\tilde{A}_{-1}$ are closed. 

By the choice of $\rho$ made in \eqref{eq:rhofix}, \eqref{eq:rhofix2} holds, namely for every $z\in \tilde{A}_{0}$ and $|z- z' |_\infty\leq\rho$ one has that $D_{\eta}(u,Q_{l}(z')) > \sigma/2$.

Moreover, since $\sigma$ satisfies \eqref{eq:gsmstr2} with $\nu=\frac1l$, when $z\in \tilde{A}_{-1}$, then one has that \[\min \bigl\{\|u-1\|_{L^1(Q_l(z))},\,\|u\|_{L^1(Q_l(z))}\bigr\} \leq  l^{d-1}.
\]

Thus, using  Lemma~\ref{lemma:stimaa-1} with $\nu=1/l$, one has that
\begin{equation*}
\begin{split}
\bar{F}_{\tau,\eps}(u, Q_{l}(z)) \geq  -C_1\frac{\Upsilon d}{\eta_0l}.
\end{split}
\end{equation*}
Moreover, let now $z'$ such that $|z- z' |_\infty\leq 1$ with $z\in \tilde{A}_{-1}$. It is not difficult to see that if $\|u-1\|_{L^1(Q_l(z))} \leq l^{d-1}$ then $\|u-1\|_{L^1(Q_l(z'))} \leq\tilde C_d l^{d-1}$. Thus from Lemma~\ref{lemma:stimaa-1}, one has that
\begin{equation}
\label{eq:tildeC}
\begin{split}
\bar{F}_{\tau,\eps}(u, Q_{l}(z')) \geq -\frac{\hat C_d\Upsilon}{\eta_0l},
\end{split}
\end{equation}
where $\hat C_d=dC_1\tilde C_d$.

The above observations motivate the following definitions
\begin{align}
A_{0} &:= \insieme{ z' \in [0,L)^d: \exists\, z \in \tilde{A}_{0}\text{ with }|z-z'|_{\infty}  \leq \rho }\label{a0}\\
A_{-1} &:= \insieme{ z' \in [0,L)^d: \exists\, z \in \tilde{A}_{-1}\text{ with }|z-z' |_{\infty}  \leq 1 },\label{a1}
\end{align}

By the choice of the parameters and the observations above, for every $z\in A_{0}$ one has that $\bar{F}_{\tau,\eps}(u,Q_{l}(z)) > M$ and for every $z\in A_{-1}$, $\bar{F}_{\tau,\eps}(u,Q_{l}(z)) \geq-(\hat C_d\Upsilon)/(l\eta_0)$.

For simplicity of notation let us denote by $A:= A_{0}\cup A_{-1}$. 

The set $[0,L)^d\setminus A$ has the following property: for every $z\in [0,L)^d\setminus A$, there exists $i\in \{ 1,\ldots,d\}$ such that $D^{i}_{\eta}(u,Q_{l}(z)) \leq \sigma$ and for every $k\neq i$ one has that $D^{k}_{\eta}(u,Q_{l}(z)) > \sigma$.

Given  that $A$ is closed, we consider the connected components $\mathcal C_{1},\ldots,\mathcal C_{n}$ of $[0,L)^d\setminus A$.  The sets $\mathcal C_{i}$ are path-wise connected. 

Let us now show the following claim: given a connected component $\mathcal C_{j}$ one has that there exists  $i$ such that $D^{i}_{\eta}(u,Q_{l}(z)) \leq \sigma$ for every $z\in\mathcal  C_{j}$  and for every $k\neq i$ one has that $D^{k}_{\eta}(u,Q_{l}(z)) > \sigma$.  Indeed, suppose that there exists $z,z'\in\mathcal  C_{j}$  such that $D^{i}_{\eta}(u,Q_{l}(z)) \leq \sigma$ and $D^{k}_{\eta}(u,Q_{l}(z')) \leq \sigma$ with $i\neq k$ and take a continuous path $\gamma:[0,1]\to \mathcal C_{j}$ such that $\gamma(0) = z$ and $\gamma(1)= z'$.
From our hypothesis, we have that $\{ s: D^{i}_{\eta}(u,Q_{l}(\gamma(s)))\leq\sigma\} \neq \emptyset$ and there exists $\tilde{s}\in \partial \{s:\, D^{i}_{\eta}(u,Q_{l}(\gamma(s))) \leq \sigma \}\cap\partial\{s:\,D^{j}_{\eta}(u,Q_{l}(\gamma(s))) \leq \sigma\}$ for some $j\neq i$. Let $\tilde{z} = \gamma(\tilde{s})$. 
Thus there are points arbitrary close to $\tilde{z}$ in $\mathcal C_{j}$ such that $D^{j}_{\eta}(u,Q_{l}(\cdot)) \leq \sigma$ and $D^{i}_\eta(u,Q_{l}(\cdot)) \leq \sigma$. 
From the continuity of the maps $z\mapsto D^{i}_\eta(u,Q_{l}(z))$, $z\mapsto D^{j}_\eta(u,Q_{l}(z))$, we have that $\tilde{z}\in A_{-1} $, which contradicts our assumption. 
We will say that $\mathcal C_j$ is oriented in direction $e_i$ if there is a point in $z\in \mathcal C_j$ such that $D^{i}_\eta(u,Q_{l}(z)) \leq \sigma$. 
Because of the above being oriented along direction $e_{i}$ is well-defined.

We will denote by $A_{i}$ the union of the connected  components $\mathcal C_{j}$ such that $\mathcal C_{j}$ is oriented along the direction $e_{i}$.

Let us now summarize the important properties that will be used in the following

\begin{enumerate}[(i)]
	\item The sets $A=A_{-1}\cup A_{0}$, $A_{1}$, $A_{2}$, $\ldots, A_d$  form a partition of $[0,L)^d$. 
	\item The sets $A_{-1}, A_{0}$ are closed and $A_{i}$, $i>0$, are open.  
	\item For every $z\in A_{i}$, we have that $D^{i}_{\eta}(u,Q_{l}(z)) \leq \sigma$. 
	\item  There exists $\rho$ (independent of $L,\tau$) such that  if $z\in A_{0}$, then $\exists\,z'$ s.t. $Q_{\rho}(z')\subset A_{0}$ and $z \in Q_{\rho}(z')$. If $z\in A_{-1}$ then $\exists\,z'$ s.t. $Q_{1}(z')\subset A_{-1}$ and $z \in Q_{1}(z')$. 
	\item For every $z\in {A}_{i}$ and $z'\in {A}_{j}$ one has that there exists a point $\tilde{z}$ in the segment connecting $z$ to $z'$ lying in ${A}_{0}\cup A_{-1}$. 
\end{enumerate}

The proof of Theorem \ref{thm:main} can thus be reduced to the following

\begin{proposition}\label{prop:final}
Let $B = \bigcup_{i> 0}A_{i}$.	For every $i$ and for our choices of $\tau$ and $\eps$, it holds 
\begin{equation}
	\label{eq:toBeShown_integral}
	\begin{split}
		\frac{1}{L^d}\int_{B} \bar{F}_{i,\tau,\eps}(u,Q_{l}(z))\dz + \frac1{d L^d} \int_{A}\bar{F}_{\tau,\eps}(u,Q_{l}(z)) \dz  \geq \frac{C^{*}_{\tau,\eps}|A_{i}|}{L^d} - C(d,\upsilon,\eta_0) \frac{|A|}{l L^d}
	\end{split}
\end{equation}
for some constant $C(d,\Upsilon,\eta_0)$ {depending on the dimension $d$, on $\Upsilon$ and $\eta_0$}.
\end{proposition}

Indeed, assuming  \eqref{eq:toBeShown_integral}, we can sum over $i$ and obtain that, since $l$ satisfies \eqref{eq:lfix} 
\begin{equation*}
\begin{split}
\Fcal_{\tau,L,\eps}(u) & \geq \sum_{i=1}^{d}\frac{1}{L^d} \int_{[0,L)^d} \bar{F}_{i,\tau,\eps}(u,Q_{l}(z))  \dz \geq  \frac{C^{*}_{\tau,\eps}}{L^d} \sum_{i=1}^{d} |A_i|  - \frac{dC(d,\Upsilon,\eta_0)|A|}{lL^d}
\\ & \geq C^{*}_{\tau,\eps} - C^*_{\tau,\eps} \frac{|A|} {L^d} - \frac{dC(d,\Upsilon,\eta_0)}{lL^d} |A| \geq C^{*}_{\tau,\eps},
\end{split}
\end{equation*}
where in the above $C^{*}_{\tau,\eps}$ is the energy density of optimal stripes of stripes $h^{*}_{\tau,\eps}$ and  we have used that $C^{*}_{\tau,\eps} < 0$ and that $|A | + \sum_{i=1}^{d} |A_{i}| =  |[0,L)^d | = L^d$. 

Notice that, in the inequality above, equality holds only if $|A|=0$ and therefore by $(v)$ only if there is just one $A_i$, $i>0$ with $|A_i|>0$. Therefore, it has been proved that there exists $i>0$ with $A_i = [0,L)^d$. 
Finally, let us consider 
\begin{align}
	\frac{1}{L^d}\int_{[0,L)^d}\bar F_{\tau,\eps}(u,Q_l(z))\dz&=\frac{1}{L^d}\int_{[0,L)^d}\bar F_{i,\tau,\eps}(u,Q_l(z))\dz\label{eq:fi}\\
	&+\frac{1}{L^d}\sum_{j\neq i}\int_{[0,L)^d}\bar F_{j,\tau,\eps}(u,Q_l(z))\dz\label{eq:fj}
\end{align}
{
We will now apply Lemma~\ref{lemma:stimalinea} with $j =i$ and slice the cube $[0,L)^d$ in direction $e_i$. 
From \eqref{eq:gstr21}, one has that \eqref{eq:fj} is nonnegative and strictly positive unless the function $u$  is one-dimensional and of the form $u(x)=u(x_j^\perp)$. 
On the other hand, from \eqref{eq:gstr28}, one has the minimum value of the \rhs of \eqref{eq:fi} is attained by one-dimensional  periodic functions  of the form $u=g(x_i)$, with period  $2h^*_{\tau,\eps}$ and satisfying \eqref{eq:refl} with $h=h^*_{\tau,\eps}$. Being $j\neq i$, the two conditions for minimizing \eqref{eq:fj} and \eqref{eq:fi} are compatible and thus minimizers of the functional $\Fcal_{\tau,L,\varepsilon}$ are as in the statement of Theorem \ref{thm:main}. }

The rest of this section is devoted to the proof of Proposition \ref{prop:final}.
\begin{proof}[Proof of Proposition \ref{prop:final}:] 
First of all notice that \eqref{eq:toBeShown_integral} follows from the analogous statement on the slices, namely that for every $t^{\perp}_{i}\in [0,L)^{d-1}$, it holds

\begin{equation}
\label{eq:toBeShown_slice}
\begin{split}
\frac{1}{L^d} \int_{B_{t^{\perp}_{i}}} \bar{F}_{i,\tau,\eps}(u,Q_{l}(t^{\perp}_{i}+se_i))\ds + \frac1{dL^d} \int_{A_{t^{\perp}_{i}}}\bar{F}_{\tau,\eps}(u,Q_{l}(t^{\perp}_{i}+se_i)) \ds  \geq \frac{C^{*}_{\tau,\eps}|A_{i,t^{\perp}_{i}}|}{L^d} - C(d,\Upsilon,\eta_0) \frac{|A_{t^\perp_i}|}{l L^d}
\end{split}
\end{equation}

Indeed by integrating \eqref{eq:toBeShown_slice} over $t^{\perp}_{i}$ we obtain \eqref{eq:toBeShown_integral}.

Notice also that $B_{t^{\perp}_{i}}$ is a finite union of intervals. Indeed, being a union of intervals follows from (ii) and finiteness follows from  condition  (v) on the decomposition.  Indeed, for every point that does not belong to $B_{t^\perp_{i}}$ because of (iv) there is a neighbourhood of fixed positive size that is not included in $B_{t^\perp_i}$. 
Let $\{ I_{1},\ldots,I_{n}\}$ such that $\bigcup_{j=1}^{n} I_{j} = B_{t_i^\perp}$ with $I_j \cap I_{k} = \emptyset$ whenever $j\neq k$. 
We can further assume that $I_{i} \leq I_{i+1}$, namely that for every $s\in I_{i}$ and $s'\in I_{i+1}$ it holds $s \leq s'$. 
By construction there exists $J_{j} \subset A_{t^{\perp}_{i}}$ such that $I_{j}\leq  J_{j} \leq I_{j+1}$.  

Whenever $J_j \cap A_{0,t_i^\perp}\neq\emptyset$,  we have that $|J_j | > \rho$  and whenever $J_{j} \cap A_{-1,t^\perp_i}\neq \emptyset $ then $|J_{i}| > 1$.

Thus we have that
\begin{equation*}
\begin{split}
\frac{1}{L^d} \int_{B_{t^{\perp}_{i}}} \bar{F}_{i,\tau,\eps}(u,Q_{l}(t^{\perp}_{i}+se_i)) \ds & + 
\frac{1}{d L^d} \int_{A_{t^{\perp}_{i}}} \bar{F}_{\tau,\eps}(u,Q_{l}(t^{\perp}_{i}+se_i)) \ds 
\\ & \geq \sum_{j=1}^n  \frac{1}{L^d}\int_{I_{j}} \bar{F}_{i,\tau,\eps}(u,Q_{l}(t^{\perp}_{i}+se_i)) \ds
+ \frac{1}{dL^d}\sum_{j=1}^n \int_{J_{j}} \bar{F}_{\tau,\eps}(u,Q_{l}(t^{\perp}_i+se_i)) \ds 
\\ & \geq \frac{1}{L^d}\sum_{j=1}^n \Big( \int_{I_{j}} \bar{F}_{i,\tau,\eps}(u,Q_{l}(t^{\perp}_{i}+se_i)) \ds
+ \frac{1}{2d} \int_{J_{j-1}\cup J_j} \bar{F}_{\tau,\eps}(u,Q_{l}(t^{\perp}_i+se_i)) \ds\Big),
\end{split}
\end{equation*}
where in order to obtain the third line from the second line, we have used periodicity and $J_0:=J_n$.

Let first $I_{j} \subset A_{i,t_i^\perp}$.  
By construction, we have that $\partial I_{j}\subset A_{t^\perp_i}$. 

If $\partial I_{j}\subset A_{-1,t^\perp_i}$, by using our choice of parameters we can apply \eqref{eq:gstr27} in Lemma~\ref{lemma:stimalinea} and obtain
\begin{equation*}
\begin{split}
\frac{1}{L^d}\int_{I_{j}} \bar{F}_{i,\tau,\eps}(u,Q_{l}(t^{\perp}_{i}+se_i))\ds  \geq \frac{1}{L^d}\Big(| I_j| C^{*}_{\tau,\eps} -\frac{M_0} l\Big).
\end{split}
\end{equation*}

If $\partial I_j \cap A_{0, t^\perp_i}\neq \emptyset$, by using our choice of parameters, we can apply \eqref{eq:gstr36} in Lemma~\ref{lemma:stimalinea}, and obtain
\begin{equation*}
\begin{split}
\frac{1}{L^d}\int_{I_{j}} \bar{F}_{i,\tau,\eps}(u,Q_{l}(t^{\perp}_{i}+se_i))\ds \geq\frac{1}{L^d}\Big(| I_j| C^{*}_{\tau,\eps}-M_0 l\Big).
\end{split}
\end{equation*}

On the other hand, if $\partial I_j \cap A_{0,t^\perp_i}\neq \emptyset$, we have that either $J_{j}\cap A_{0,t^\perp_i}\neq \emptyset$ or $J_{j-1}\cap A_{0,t^\perp_i}\neq\emptyset$. Thus
\begin{equation*}
\begin{split}
\frac{1}{2dL^d}\int_{J_{j-1}} \bar{F}_{\tau,\eps}(u,Q_{l}(t^{\perp}_{i}+se_i)) \ds & + \frac{1}{2dL^d}\int_{J_{j}} \bar{F}_{\tau,\eps}(u,Q_{l}(t^{\perp}_{i}+se_i)) \ds  \\ &\geq  \frac{M\rho}{2dL^d}  - \frac{|J_{j-1}\cap A_{-1,t^\perp_i} |\hat C_d\Upsilon}{2d\eta_0l L^d} - \frac{|J_{j}\cap A_{-1,t^\perp_i} |\hat C_d\Upsilon}{2d\eta_0l L^d},
\end{split}
\end{equation*}
where $\hat C_d$ is the  constant in \eqref{eq:tildeC}.

Since $M$ satisfies \eqref{eq:Mfix}, in both cases $\partial I_{j}\subset A_{-1,t^\perp_i}$ or $\partial I_{j}\cap A_{0,t^\perp_i}\neq \emptyset$, we have that 
\begin{equation*}
\begin{split}
\frac{1}{L^d}\int_{I_{j}} \bar{F}_{i,\tau,\eps}(u,Q_{l}(t^\perp_i+se_i)) \ds &+ 
\frac{1}{2dL^d}\int_{J_{j-1}}  \bar{F}_{\tau,\eps}(u,Q_{l}(t^{\perp}_{i}+se_i))\ds
+ \frac{1}{2dL^d}\int_{J_{j}}  \bar{F}_{\tau,\eps}(u,Q_{l}(t^{\perp}_{i}+se_i))\ds\\
&\geq \frac{C^{*}_{\tau,\eps} |I_{j} |}{L^d} - \frac{|J_{j-1}\cap A_{-1,t^\perp_i}|\hat C_d\Upsilon}{2d\eta_0lL^d}
- \frac{|J_{j}\cap A_{-1,t^\perp_i}|\hat C_d\Upsilon}{2d\eta_0lL^d}.
\end{split}
\end{equation*}

If $I_{j} \subset A_{k,t^\perp_i}$ with $k \neq i$ from \eqref{eq:gstr20} in  Lemma~\ref{lemma:stimalinea} it holds
\begin{equation*}
\begin{split}
\frac 1{L^d}\int_{I_{j}} \bar{F}_{i,\tau,\eps}(u,Q_{l}(t^{\perp}_{i}+se_i)) \ds\geq  - \frac{M_0}{lL^d}.
\end{split}
\end{equation*}

In general for every $J_{j}$  we have that 
\begin{equation*}
\begin{split}
\frac{1}{dL^d}\int_{J_{j}} \bar{F}_{\tau,\eps}(u,Q_{l}(t^{\perp}_{i}+se_i))\, \ds \geq   \frac{|J_{j}\cap A_{0,t^\perp_i} | M}{dL^d} - \frac{\hat C_d\Upsilon}{d\eta_0lL^d }|J_{j}\cap A_{-1,t^\perp_i}|. 
\end{split}
\end{equation*}

For $I_{j}\subset A_{k,t^\perp_i}$ such that $(J_j \cup J_{j-1})\cap A_{0,t^\perp_i}\neq \emptyset$ with $k\neq i$, we have that 
\begin{equation*}
\begin{split}
\frac{1}{L^d}\int_{I_{j}} \bar{F}_{i,\tau,\eps}(u,Q_{l}(t^\perp_i+se_i)) \ds &+ 
\frac{1}{2dL^d}\int_{J_{j-1}}  \bar{F}_{\tau,\eps}(u,Q_{l}(t^{\perp}_{i}+se_i))\ds
+ \frac{1}{2dL^d}\int_{J_{j}}  \bar{F}_{\tau,\eps}(u,Q_{l}(t^{\perp}_{i}+se_i))\ds\\
&\geq -\frac{M_{0}}{lL^d} + \frac{M\rho}{2dL^d} - \frac{|J_{j-1}\cap A_{-1,t^\perp_i}|\hat C_d\Upsilon}{2d\eta_0lL^d}
- \frac{|J_{j}\cap A_{-1,t^\perp_i}|\hat C_d\Upsilon}{2d\eta_0lL^d}.
\\ &\geq
- \frac{|J_{j-1}\cap A_{-1,t^\perp_i}|\hat C_d\Upsilon}{2d\eta_0lL^d}
- \frac{|J_{j}\cap A_{-1,t^\perp_i}|\hat C_d\Upsilon}{2d\eta_0lL^d}.
\end{split}
\end{equation*}
where the last inequality is true due to \eqref{eq:Mfix}.

For $I_{j}\subset A_{k,t^\perp_i}$ such that $(J_j \cup J_{j-1})\subset  A_{-1,t^\perp_i}$ with $k\neq i$, we have that 
\begin{equation*}
\begin{split}
\frac{1}{L^d}\int_{I_{j}} \bar{F}_{i,\tau,\eps}(u,Q_{l}(t^\perp_i+se_i)) \ds &+ 
\frac{1}{2dL^d}\int_{J_{j-1}}  \bar{F}_{\tau,\eps}(u,Q_{l}(t^{\perp}_{i}+se_i))\ds
+ \frac{1}{2dL^d}\int_{J_{j}}  \bar{F}_{\tau,\eps}(u,Q_{l}(t^{\perp}_{i}+se_i))\ds\\
&\geq -\frac{M_{0}}{lL^d}   - \frac{|J_{j-1}\cap A_{-1,t^\perp_i}|\hat C_d\Upsilon}{2d\eta_0lL^d}
- \frac{|J_{j}\cap A_{-1,t^\perp_i}|\hat C_d\Upsilon}{2d\eta_0lL^d}.
\\ &\geq
- \max\Big(M_{0},\frac{\hat C_d\Upsilon}{d\eta_0}\Big)\bigg(\frac{|J_{j-1}\cap A_{-1,t^\perp_i}|}{lL^d}
+ \frac{|J_{j}\cap A_{-1,t^\perp_i}|}{lL^d}\bigg).
\end{split}
\end{equation*}
where in the last inequality we have used that $|J_j\cap A_{-1,t^\perp_i}|\geq1, \,|J_{j-1}\cap A_{-1,t^\perp_i}|\geq1$.

Summing over $j$, and taking $C(d,\Upsilon,\eta_0)=\max\Big(M_{0},\frac{\hat C_d\Upsilon}{d\eta_0}\Big) $, one obtains \eqref{eq:toBeShown_slice} as desired.

\end{proof}

\end{document}